\documentclass[11pt]{article}

\makeatletter
\newcommand\RedeclareMathOperator{%
  \@ifstar{\def\rmo@s{m}\rmo@redeclare}{\def\rmo@s{o}\rmo@redeclare}%
}
\newcommand\rmo@redeclare[2]{%
  \begingroup \escapechar\m@ne\xdef\@gtempa{{\string#1}}\endgroup
  \expandafter\@ifundefined\@gtempa
     {\@latex@error{\noexpand#1undefined}\@ehc}%
     \relax
  \expandafter\rmo@declmathop\rmo@s{#1}{#2}}
\newcommand\rmo@declmathop[3]{%
  \DeclareRobustCommand{#2}{\qopname\newmcodes@#1{#3}}%
}
\@onlypreamble\RedeclareMathOperator
\makeatother

\usepackage{witharrows}
\usepackage{tikz}
\usepackage{graphicx,wrapfig,lipsum}   

\usepackage{xspace}

\usepackage{framed}

\usepackage{amsmath,amssymb}
\usepackage{amsthm}
\usepackage{mathtools}
\usepackage[noend]{algorithmic}
\usepackage[ruled,vlined]{algorithm2e}
\usepackage{url}
\usepackage{fullpage}
\usepackage{makeidx}
\usepackage{enumerate}
\usepackage[top=1in, bottom=1.25in, left=0.9in, right=0.9in]{geometry}
\usepackage{graphicx,float,psfrag,epsfig,caption}

\usepackage{hyperref}
\usepackage{epstopdf}
\usepackage{color}
\usepackage{xr}
\usepackage{subcaption}
\usepackage[utf8]{inputenc}

\usepackage{scalerel,stackengine}

\usepackage{thm-restate}

\usepackage{enumitem}
\usepackage{bbm}

\stackMath
\newcommand\reallywidehat[1]{%
\savestack{\tmpbox}{\stretchto{%
  \scaleto{%
    \scalerel*[\widthof{\ensuremath{#1}}]{\kern.1pt\mathchar"0362\kern.1pt}%
    {\rule{0ex}{\textheight}}
  }{\textheight}%
}{2.4ex}}%
\stackon[-6.9pt]{#1}{\tmpbox}%
}
\usepackage[mathscr]{euscript} 

\DeclareSymbolFont{rsfs}{U}{rsfs}{m}{n}
\DeclareSymbolFontAlphabet{\mathscrsfs}{rsfs}
\usepackage{mathrsfs}

\numberwithin{equation}{section}

\renewcommand{\paragraph}[1]{\noindent\textbf{#1}.\quad}

\newtheoremstyle{myexample} 
    {\topsep}                    
    {\topsep}                    
    {\rm }                   
    {}                           
    {\bf }                   
    {.}                          
    {.5em}                       
    {}  

\newtheoremstyle{myremark} 
    {\topsep}                    
    {\topsep}                    
    {\rm}                        
    {}                           
    {\bf}                        
    {.}                          
    {.5em}                       
    {}  

\newtheorem{claim}{Claim}[section]
\newtheorem{lemma}[claim]{Lemma}

\newtheorem{assumption}{Assumption}[]

\newtheorem{theorem}{Theorem}
\newtheorem{proposition}[claim]{Proposition}
\newtheorem{corollary}[claim]{Corollary}

\theoremstyle{definition}
\newtheorem{definition}[claim]{Definition}

\theoremstyle{myremark}
\newtheorem{remark}{Remark}[section]

\theoremstyle{myremark}

\theoremstyle{myexample}

\definecolor{darkgreen}{rgb}{0.0, 0.5, 0.0}
\usepackage[suppress]{color-edits}  
\addauthor{rev}{red}   
\addauthor{bh}{blue} 
\addauthor{ms}{orange}

\hypersetup{
  colorlinks   = true, 
  urlcolor     = blue, 
  linkcolor    = blue, 
  citecolor   = red 
}

\newcommand{\bea}{\begin{eqnarray}}
\newcommand{\eea}{\end{eqnarray}}
\newcommand{\<}{\langle}
\renewcommand{\>}{\rangle}

\newcommand{\wt}{\widetilde}
\newcommand{\op}{\text{op}}
\newcommand{\wh}{\widehat}

\def\I{{\rm I}}
\def\II{{\rm II}}

\def\eps{{\varepsilon}}

\def\id{{\boldsymbol{I}}}
\def\bh{\boldsymbol{h}}

\def\btheta{{\boldsymbol{\theta}}}
\def\bTheta{{\boldsymbol{\Theta}}}

\def\bferr{{\boldsymbol{err}}}

\def\bZ{{\boldsymbol{Z}}}
\def\bW{{\boldsymbol{W}}}
\def\bE{{\boldsymbol{E}}}

\def\bP{{\boldsymbol{P}}}

\def\bc{{\boldsymbol{c}}}
\def\bC{{\boldsymbol{C}}}
\def\bQ{{\boldsymbol{Q}}}

\def\bV{{\boldsymbol{V}}}

\def\bg{{\boldsymbol{g}}}

\def\bx{{\boldsymbol{x}}}

\def\bzero{{\mathbf 0}}
\def\bone{{\mathbf 1}}
\def\cF{{\mathcal F}}

\def\cT{{\mathcal T}}

\def\blambda{{\boldsymbol \lambda}}

\def\sS{{\mathscr S}}

\def\op{\mbox{\tiny\rm op}}

\def\vu{\vec u}
\def\vv{\vec v}
\def\vw{\vec w}
\def\vx{\vec x}

\def\vA{\vec A}

\def\vC{\vec C}

\def\wtH{\wt{H}}

\def\vone{\vec 1}
\def\vzero{\vec 0}

\def\bsig{{\boldsymbol {\sigma}}}

\def\vh{{\vec h}}

\def\vlam{{\vec \lambda}}
\def\valpha{{\vec \alpha}}
\def\va{{\vec a}}
\def\vb{{\vec b}}

\def\brho{{\boldsymbol \rho}}

\def\by{{\boldsymbol y}}
\def\vy{{\vec y}}

\def\naturals{{\mathbb N}}
\def\reals{{\mathbb R}}

\def\normal{{\sf N}}

\def\sT{{\sf T}}

\def\bv{{\boldsymbol{v}}}
\def\bz{{\boldsymbol{z}}}
\def\bx{{\boldsymbol{x}}}

\def\bA{\boldsymbol{A}}

\def\bT{\boldsymbol{T}}
\def\bm{\boldsymbol{m}}

\def\va{\vec{a}}
\def\vb{\vec{b}}

\def\ons{\mathbf{ons}}
\def\ONS{\mathbf{ONS}}

\def\de{{\rm d}}

\def\bY{\boldsymbol{Y}}
\def\bW{\boldsymbol{W}}

\def\<{\langle}
\def\>{\rangle}

\def\diag{{\rm diag}}
\def\ed{\stackrel{{\rm d}}{=}}

\def\cN{{\cal N}}

\def\cV{{\cal V}}
\def\cL{{\cal L}}

\def\by{{\boldsymbol{y}}}

\def\bw{{\boldsymbol{w}}}
\def\P{\mathbb{P}}

\def\be{{\boldsymbol{e}}}
\def\blambda{{\boldsymbol{\lambda}}}
\def\bD{{\boldsymbol{D}}}

\def\bu{{\boldsymbol{u}}}
\def\b0{{\boldsymbol{0}}}

\def\bfone{{\boldsymbol 1}}

\def\bG{{\boldsymbol G}}

\DeclareMathOperator*{\plim}{p-lim}

\def\OPT{{\sf OPT}}

\def\cA{{\mathcal A}}
\def\cI{{\mathcal I}}
\def\cS{{\mathcal S}}
\def\bn{{\boldsymbol n}}
\def\cB{{\mathcal B}}

\def\bzero{\boldsymbol{0}}

\def\ul{\underline{\ell}}
\def\ol{\overline{\ell}}

\newcommand{\f}{\boldsymbol{f}}
\renewcommand{\b}{\mathbf{b}}

\def\fr{\frac}
\def\lt{\left}
\def\rt{\right}

\def\la{\langle}
\def\ra{\rangle}

\def\eps{\varepsilon}

\def\bbA{{\mathbb{A}}}

\def\bbE{{\mathbb{E}}}

\def\bbP{{\mathbb{P}}}
\def\bbR{{\mathbb{R}}}

\def\bbW{{\mathbb{W}}}
\def\bbZ{{\mathbb{Z}}}

\def\cA{{\mathcal{A}}}
\def\cB{{\mathcal{B}}}
\def\cF{{\mathcal{F}}}

\def\cN{{\mathcal{N}}}
\def\cP{{\mathcal{P}}}

\def\sH{{\mathscr{H}}}

\def\bg{{\mathbf{g}}}
\def\bh{{\boldsymbol{h}}}

\def\bq{{\mathbf{q}}}

\def\bP{\mathbf{P}}
\def\bQ{\mathbf{Q}}

\def\ALG{{\mathsf{ALG}}}

\def\OPT{{\mathsf{OPT}}}

\DeclareMathOperator*{\E}{\bbE}

\newcommand{\Gp}[1]{\mathbf{G}^{(#1)}}

\newcommand{\norm}[1]{{\lt\|#1\rt\|}}
\newcommand{\tnorm}[1]{{\|#1\|}}

\newcommand{\diff}[1]{{\mathrm{d}#1}}
\newcommand{\deriv}[1]{{\fr{\diff{}}{\diff{#1}}}}

\def\vR{\vec R}

\def\ons{\mathbf{ons}}
\def\ONS{\mathbf{ONS}}

\def\ZZ{{\boldsymbol{Z}}}
\def\W{{\boldsymbol{W}}}
\def\A{{\boldsymbol{A}}}

\def\hZZ{\widehat{\boldsymbol{Z}}}
\def\AMP{{\sf AMP}}
\def\LAMP{{\sf LAMP}}

\def\qq{{\boldsymbol{q}}}

\def\VV{{\boldsymbol{V}}}

\def\sph{{\mathrm{sp}}}
\def\rd{{\mathrm{rad}}}

\def\spn{{\rm span}}

\def\hbA{\wh{\boldsymbol A}}
\def\tbA{\wt{\boldsymbol A}}
\def\hbz{\hat{\boldsymbol z}}

\def\bfone{{\boldsymbol 1}}

\def\bfzero{\boldsymbol{0}}

\def\vDelta{{\vec\Delta}}

\title{Optimization Algorithms for Multi-Species Spherical Spin Glasses}
\author{
    Brice Huang\thanks{Department of Electrical Engineering and Computer Science, Massachusetts Institute of Technology. Email: \texttt{bmhuang@mit.edu}.} 
    \and 
    Mark Sellke\thanks{Department of Statistics, Harvard University.
    Email: \texttt{msellke@fas.harvard.edu}.}
}

\date{}

\begin{document}

\maketitle

\begin{abstract}
    This paper develops approximate message passing algorithms to optimize multi-species spherical spin glasses. We first show how to efficiently achieve the algorithmic threshold energy identified in our companion work \cite{huang2023algorithmic}, thus confirming that the Lipschitz hardness result proved therein is tight.
    Next we give two generalized algorithms which produce multiple outputs and show all of them are approximate critical points. 
    Namely, in an $r$-species model we construct $2^r$ approximate critical points when the external field is stronger than a ``topological trivialization" phase boundary, and exponentially many such points in the complementary regime.
    We also compute the local behavior of the Hamiltonian around each. 
    These extensions are relevant for another companion work \cite{huang2023strong} on topological trivialization of the landscape. 
\end{abstract}
\setcounter{tocdepth}{2}
\tableofcontents
\newpage

\section{Introduction}
\label{sec:intro}

This paper studies the efficient optimization of a family of random non-convex functions $H_N$ defined on high-dimensional spaces, namely the Hamiltonians of multi-species spherical spin glasses. 
Mean-field spin glasses have been studied since \cite{sherrington1975solvable} as models for disordered magnetic systems and are also closely linked to random combinatorial optimization problems \cite{krzakala2007gibbs, dembo2017extremal, panchenko2018k}. In short, their Hamiltonians are certain polynomials in many variables with independent centered Gaussian coefficients.

The purpose of this work is to develop efficient algorithms to optimize $H_N$.
Our companion work \cite{huang2023algorithmic} derives an \emph{algorithmic threshold} $\ALG$ and proves no optimization algorithm with suitably Lipschitz dependence on $H_N$ can achieve energy better than $\ALG$ with more than exponentially small probability.
The value $\ALG$ is expressed as the maximum of a variational principle over several increasing functions, which was shown to be achieved by joining the solutions to a pair of well-posed differential equations.
The first main contribution of this paper is to show that given a solution to this variational problem, so-called approximate message passing (AMP) algorithms efficiently achieve the value $\ALG$.
We note that several previous works \cite{subag2018following,mon18,ams20,sellke2021optimizing} have given similar algorithms for mean-field spin glasses with $1$ species, and our algorithm is in line with the latter three.

Furthermore, we use these AMP algorithms to aid a detailed study of the landscape of $H_N$ by probing neighborhoods of special critical points.
This is related to a second companion work \cite{huang2023strong} which identifies the phase boundary for \emph{topological trivialization} of $H_N$, where the number of critical points is a constant independent of $N$. 
Therein, Kac-Rice estimates are used to show that for $r$-species models (defined on a product of $r$ spheres) in the ``super-solvable'' regime with strong external field, $H_N$ has exactly $2^r$ critical points with high probability.
In this paper, we give a signed AMP algorithm which explicitly approximates each of these critical points.
Moreover in the complementary ``sub-solvable'' regime, we use AMP to construct $\exp(cN)$ separated approximate critical points with high probability. 
This implies the failure of \emph{strong topological trivialization} as defined in \cite{huang2023strong}, which is proved therein to hold for super-solvable models.
Finally, the machinery of AMP allows us to compute the local behavior of $H_N$ around these algorithmic outputs, giving even more precise information about the landscape. 

\subsection{Problem Description}

Fix a finite set $\sS = \{1,\ldots,r\}$. 
For each positive integer $N$, fix a deterministic partition $\{1,\ldots,N\} = \sqcup_{s\in\sS}\, \cI_s$ with $\lim_{N\to\infty} |\cI_s| / N =\lambda_s$ where $\vlam = (\lambda_1,\ldots,\lambda_r) \in \bbR_{>0}^\sS$.
For $s\in \sS$ and $\bx \in \bbR^N$, let $\bx_s \in \bbR^{\cI_s}$ denote the restriction of $\bx$ to coordinates $\cI_s$.
We consider the state space 
\begin{equation}
\label{eq:def-BN}
    \cB_N = \lt\{
        \bx \in \bbR^N : 
        \norm{\bx_s}_2^2 \le \lambda_s N
        \quad\forall~s\in \sS
    \rt\}.
\end{equation}
Fix $\vh = (h_1,\ldots,h_r) \in \bbR_{\ge 0}^\sS$ and let $\bone = (1,\ldots,1) \in \bbR^N$.
For each $k\ge 2$ fix a symmetric tensor $\Gamma^{(k)} = (\gamma_{s_1,\ldots,s_k})_{s_1,\ldots,s_k\in \sS} \in (\bbR_{\ge 0}^{\sS})^{\otimes k}$ with $\sum_{k\ge 2} 2^k \norm{\Gamma^{(k)}}_\infty < \infty$, and let $\Gp{k} \in (\bbR^N)^{\otimes k}$ be a tensor with i.i.d. standard Gaussian entries.

For $A\in (\bbR^\sS)^{\otimes k}$, $B\in (\bbR^N)^{\otimes k}$, define $A\diamond B \in (\bbR^N)^{\otimes k}$ to be the tensor with entries
\begin{equation}
    \label{eq:def-diamond}
    (A\diamond B)_{i_1,\ldots,i_k} = A_{s(i_1),\ldots,s(i_k)} B_{i_1,\ldots,i_k},
\end{equation}
where $s(i)$ denotes the $s\in \sS$ such that $i\in \cI_s$.
Let $\bh = \vh \diamond \bone$.
We consider the mean-field multi-species spin glass Hamiltonian
\begin{align}
    \label{eq:def-hamiltonian}
    H_N(\bsig) &= \la \bh, \bsig \ra + \wtH_N(\bsig), \quad \text{where}\\
    \label{eq:def-hamiltonian-no-field}
    \wtH_N(\bsig) &=
	\sum_{k\ge 2}
	\fr{1}{N^{(k-1)/2}}
	\la \Gamma^{(k)} \diamond \bG^{(k)}, \bsig^{\otimes k} \ra \\
	\notag
	&=
	\sum_{k\ge 2}
	\fr{1}{N^{(k-1)/2}}
	\sum_{i_1,\ldots,i_k=1}^N 
	\gamma_{s(i_1),\ldots,s(i_k)} \bG^{(k)}_{i_1,\ldots,i_k} \sigma_{i_1}\cdots \sigma_{i_k}
\end{align}
with inputs $\bsig = (\sigma_1,\ldots,\sigma_N) \in \cB_N$.
\revedit{For example, the choice of parameters $\Gamma^{(2)} = (\begin{smallmatrix} 0 & 1 \\ 1 & 0 \end{smallmatrix})$ and $\Gamma^{(k)}=0$ for $k\ge 3$ is the well-known bipartite spherical SK model \cite{auffinger2014free}.}
For $\bsig,\brho\in \cB_N$, define the species $s$ overlap and overlap vector
\begin{equation}
    \label{eq:R}
    R_s(\bsig, \brho)
    =
    \fr{ \la \bsig_s, \brho_s \ra}{\lambda_s N},
    \qquad
    \vR(\bsig, \brho) 
    = 
    \lt(R_1(\bsig, \brho), \ldots, R_r(\bsig, \brho)\rt).
\end{equation}
Let $\odot$ denote coordinate-wise product. 
For $\vx = (x_1,\ldots,x_r) \in \bbR^\sS$, let 
\begin{align*}
    \xi(\vx) 
    &= \sum_{k\ge 2} \la \Gamma^{(k)}\odot \Gamma^{(k)}, (\vlam \odot \vx)^{\otimes k}\ra \\
    &= \sum_{k\ge 2}
	\sum_{s_1\ldots,s_k\in \sS}
	\gamma_{s_1,\ldots,s_k}^2
	(\lambda_{s_1} x_{s_1})
	\cdots
	(\lambda_{s_k} x_{s_k}).
\end{align*}
The random function $\wtH_N$ can also be described as the Gaussian process on $\cB_N$ with covariance
\[
	\bbE \wtH(\bsig)\wtH(\brho)
	=
	N\xi(\vR(\bsig, \brho)).
\]
\revedit{
We will also often refer to the product of spheres 
\begin{equation}
\label{eq:def-SN}
\cS_N=\big\{\bu\in\bbR^N~:~\|\bu_s\|^2=\lambda_s N~~\forall~s\in \sS\big\}.
\end{equation}
}
It will be useful to define, for $s\in \sS$,
\[
    \xi^s(\vx) 
    = 
    \lambda_s^{-1} 
    \partial_{x_s} 
    \xi(\vx).
\]

\subsection{The Value $\ALG$}

\revedit{
Given $(\vlam,\xi)$, the ground state energy of the associated multi-species spherical spin glass is\footnote{Technically the $N\to\infty$ limit is not known to exist for general $\xi$. Since $\OPT$ appears in the present paper only in this informal discussion, we will not belabor this point.}
\[
\OPT=\OPT(\xi)
=
\plim_{N\to\infty}
\sup_{\bsig\in\cB_N} H_N(\bsig)/N.
\]
In the bipartite SK model mentioned above, $\OPT$ is the limiting operator norm of an IID Gaussian rectangular matrix with aspect ratio $\lambda_1/\lambda_2$. 
For large $k$, the asymptotic operator norm of an IID random $k$-tensor is similarly encoded as $\OPT(\xi)$ for some $\xi$ (with e.g. $r=k$).
Perhaps surprisingly, it is generally believed that polynomial-time algorithms are not in general capable of finding $\bsig\in\cB_N$ such that $H_N(\bsig)\geq \OPT(\xi)-\eps$ with high probability as $N\to\infty$.
Our work \cite{huang2021tight} showed that in the single species case (and with all terms of even degree), one can identify an exact threshold $\ALG$ for the performance of a class of \emph{Lipschitz} algorithms which includes gradient-based methods and Langevin dynamics. 
More recently in \cite{huang2023algorithmic}, we extended the algorithmic hardness direction of this result to multi-species spherical spin glasses, using a new proof technique that applies even when $\OPT$ is not known.
The purpose of this paper is to give explicit algorithms attaining the value $\ALG$, and we present here the formula for this value.
}

The algorithmic threshold $\ALG$ is given by the following variational principle. 
This is a simplification of the more general variational formula \cite[Equation (1.7)]{huang2023algorithmic}, obtained by a partial characterization of its maximizers \cite[Theorem 3]{huang2023algorithmic}.
\revedit{The following generic assumption is needed therein to ensure well-posedness of the ODE \eqref{eq:tree-descending-ode} used in this description, and we will freely assume it throughout the paper.}
\begin{assumption}
    \label{as:nondegenerate}
    All quadratic and cubic interactions participate in $H$, i.e. $\Gamma^{(2)}, \Gamma^{(3)} > 0$ coordinate-wise.
    We will call such models \textbf{non-degenerate}.
    \revedit{Since this condition depends only on $\xi$, we similarly call $\xi$ non-degenerate.}
\end{assumption}

To optimize $H_N$ for degenerate $\xi$, it suffices to apply our algorithms to a slight perturbation $\wt\xi$ which is non-degenerate and satisfies $\|\xi-\wt\xi\|_{C^3([0,1]^r)}\leq \eps$ to obtain the guarantees in this and the next section. 
\revedit{Here, $C^3([0,1]^r)$ denotes the norm
\[
    \|\xi\|_{C^3([0,1]^r)} 
    = \sup_{\vx \in [0,1]^r} \max \lt\{
        |\xi(\vx)|, \|\nabla \xi(\vx)\|_\infty, \|\nabla^2 \xi(\vx)\|_\infty, \|\nabla^3 \xi(\vx)\|_\infty
    \rt\}.
\]}
Since both the ground state and the more general $\ALG$ formula in \cite{huang2023algorithmic} (allowing degenerate $\xi$) vary continuously in $\xi$, there is essentially no loss of generality in assuming non-degeneracy.

The formula for $\ALG$ is described by two cases depending on whether $\vone=1^{\sS}$ is super-solvable as defined below.

\begin{definition}
    \label{defn:diag-signed}
    A matrix $M\in \bbR^{\sS \times \sS}$ is \textbf{diagonally signed} if $M_{i,i}\ge 0$ and $M_{i,j}<0$ for all $i\neq j$.
\end{definition}
\begin{definition}
    \label{defn:solvable}
    A symmetric diagonally signed matrix $M$ is \textbf{super-solvable} if it is positive semidefinite, and \textbf{solvable} if it is furthermore singular; otherwise $M$ is \textbf{strictly sub-solvable}.
    A point $\vx \in (0,1]^\sS$ is super-solvable, solvable, or strictly sub-solvable if $M^*(\vx)$ is, where
    \begin{equation}
        \label{eq:M*sym}
        M^*(\vx) = 
        \diag\lt(\lt(\fr{\partial_{x_s}\xi(\vx) + \lambda_s h_s^2}{x_s}\rt)_{s\in \sS}\rt) 
        - \lt(\partial_{x_s,x_{s'}}\xi(\vx)\rt)_{s,s'\in \sS}
        .
    \end{equation}
    We also adopt the convention that $\vzero$ is always super-solvable, and solvable if $\vh=\vzero$. 
\end{definition}

The following will be useful.

\begin{proposition}[{\cite[Proposition 4.3]{huang2023algorithmic}, see also \cite[Lemma 2.5]{huang2023strong}}]
\label{prop:diagonally-signed-min-max}
    If the square matrix $M$ is diagonally signed, then the minimal eigenvalue $\blambda_{\min}(M)$ has multiplicity $1$, and the corresponding eigenvector $\vv$ has strictly positive entries. 
    Moreover
    \[
        \blambda_{\min}(M) 
        = \sup_{\vv \succ \vzero}
        \min_{s\in\sS}
        \fr{(M\vv)_s}{v_s}\,,
    \]
    and the supremum is uniquely attained at $\vv$.
\end{proposition}

It is easy to see that any $x\in (0,1]^\sS$ is sub-solvable when $\vh=\vzero$, and that super-solvability is a coordinate-wise increasing property of $\vh$.
For our purposes, an external field is large if $\vone$ is super-solvable and small if $\vone$ is strictly sub-solvable.
\revedit{
(Unfortunately we do not have more refined intuition for the precise form of $M^*$ above, nor the resulting phase boundary between super and sub-solvability.)
As shown in our companion work \cite{huang2023strong}, in super-solvable models the external fields $\bh$ are strong enough to trivialize the ``glassy'' nature of the landscape for $H_N$. Namely the number of critical points is exactly $2^r$ with high probability, the minimum number of any generic smooth (``Morse'') function on a product of $r$ spheres. By contrast in the sub-solvable case, the expected number of critical points is exponentially large in the dimension $N$. As explained below, the optimization algorithms are also simpler in the super-solvable case.
}

\begin{definition}[Algorithmic Threshold, Super-Solvable Case]
    If $\vone$ is super-solvable, then 
    \[
    \ALG
    =
    \sum_{s\in \sS}
    \lambda_s 
    \sqrt{\xi^s(\vone) + h_s^2}
    \,
    .
    \]
\end{definition}

When $\vone$ is strictly sub-solvable, the formula for $\ALG$ becomes more complicated and depends on the optimal choice of a increasing $C^2$ function $\Phi:[q_1,1]\to [0,1]^{\sS}$ satisfying certain conditions.
We term such $\Phi$ \emph{pseudo-maximizers} and defer the formal definition to Definition~\ref{def:pseudo-maximizer}. 
Note that $q_1 \in [0,1]$ is not fixed, but is determined by the choice of $\Phi$.
\begin{definition}[Algorithmic Threshold, Strictly Sub-solvable Case]
    \label{def:subsolvable-alg}
    If $\vone$ is strictly sub-solvable, then with the maximum taken over all pseudo-maximizers $\Phi$ of $\bbA$,
    \begin{equation}
    \label{eq:alg-for-optimizer}
    \begin{aligned}
    \ALG
    &=
    \max_{\Phi}
    \bbA(\Phi);
    \\
    \bbA(\Phi)
    &\equiv
    \sum_{s\in \sS}
    \lambda_s \lt[
        \sqrt{\Phi_s(q_1) (\xi^s(\Phi(q_1)) + h_s^2)}  + 
        \int_{q_1}^1 \sqrt{\Phi'_s(q)(\xi^s\circ \Phi)'(q)}~\de q
    \rt]
    \,
    .
    \end{aligned}
    \end{equation}
\end{definition}

\revedit{See \cite[Remark 1.3]{huang2023algorithmic} for an approach to maximizing $\bbA$ using the well-posedness of the ODEs \eqref{eq:Phip-q1}, \eqref{eq:tree-descending-ode} in the definition of pseudo-maximizer. The computational complexity of this task is in particular independent of $N$.}

The following theorem is our main result. 
We equip the space $\sH_N$ of Hamiltonians $H_N$ with the following distance. 
We identify $H_N$ with its disorder coefficients $(\bG^{(k)})_{k\ge 2}$, which we \revedit{arrange} in an arbitrary but fixed order into an infinite vector $\bg(H_N)$, and define
\[
    \norm{H_N-H'_N}_2 = \norm{\bg(H_N) - \revedit{\bg}(H'_N)}_2.
\]
\revedit{(In other words, $\norm{H_N-H'_N}_2^2$ is the sum of squared differences $(g_{i_1,\dots,i_k}-g'_{i_1,\dots,i_k})^2$ between all corresponding pairs of coefficients in $(\bG^{(k)})_{k\ge 2}$ and $(\bG'^{(k)})_{k\ge 2}$.)}
We say an algorithm $\cA_N : \sH_N \to \cB_N$ is $\tau$-Lipschitz if 
\[
    \norm{\cA_N(H_N) - \cA_N(H'_N)}_2 \le \tau 
    \norm{H_N - H'_N}_2, \qquad 
    \forall H_N, H'_N \in \sH_N.
\]
Note that $\norm{H_N-H'_N}_2$ may be infinite, and if so this condition holds vacuously for such pairs $(H_N,H'_N)$.
Here and throughout, all implicit constants may depend also on $(\xi,\vh,\vlam)$.

\begin{theorem}
\label{thm:main-alg}
For any $\eps>0$, there exists an $O_{\eps}(1)$-Lipschitz $\cA_N:\sH_N\to \cB_N$ such that
\[
    \bbP[H_N(\cA_N(H_N))/N \geq \ALG-\eps]
    \ge 1-\exp(-cN),
    \quad 
    c = c(\eps) > 0.
\] 
\end{theorem}

The main result in our companion work \cite[Theorem 1]{huang2023algorithmic} states that any $\tau$-Lipschitz $\cA_N : \sH_N \to \cB_N$ satisfies, for the same threshold $\ALG$ and $N$ sufficiently large,
\[
    \bbP[H_N(\cA_N(H_N))/N \ge \ALG + \eps]
    \le 
    \exp(-cN),
    \quad
    c = c(\eps,\tau) > 0.
\]
Together these results thus characterize the best possible Lipschitz optimization algorithms for multi-species spherical spin glasses.

We prove Theorem~\ref{thm:main-alg} with an explicit algorithm based on approximate message passing (AMP), following a recent line of work \cite{subag2018following,mon18,ams20,alaoui2022algorithmic,sellke2021optimizing}.
Such algorithms are shown to be Lipschitz (up to modification on a set with $\exp(-cN)$ probability) in \cite[Section 8]{huang2021tight}.
AMP algorithms also have computational complexity which is linear in the input size when $H_N$ is a polynomial of finite degree (modulo solving for $\Phi$, a task that does not depend on $N$). 
See \cite[Remark 2.1]{ams20} for related discussion on this last point.

Similarly to \cite{alaoui2022algorithmic,sellke2021optimizing}, our algorithm has two phases, a ``root-finding" phase and a ``tree-descending" phase. 
Roughly speaking, the set of points reachable by our algorithm has the geometry of a densely branching ultrametric tree, which is rooted at the origin when $\bh = \bzero$ and more generally at a random point correlated with $\bh$.
The first phase identifies this root, and the second traces a root-to-leaf path of this tree. 
The structure of the first phase is similar to the original AMP algorithm of \cite{bolthausen2014iterative} for the SK model at high-temperature, while the latter \emph{incremental} AMP technique was introduced in \cite{mon18}.

\revedit{For the purposes of this paper, the significance} of (super, sub)-solvability is as follows.
When the external field is sufficiently large, the root moves all the way to the boundary of $\cB_N$ \revedit{(in all $r$ species)} and the algorithmic tree becomes degenerate.
In \cite{huang2023algorithmic}, it is shown that the external field is large enough for this to occur if and only if $\vone$ is super-solvable. 
Moreover, \cite{huang2023strong} shows this condition coincides with strong topological trivialization (defined therein) of the optimization landscape.

In Section~\ref{sec:branch} we extend our main algorithm in several ways. 
In Subsection~\ref{subsec:signed} we define $2^r$ signed generalizations of the root-finding algorithm with similar behavior.
In Subsection~\ref{subsec:grad} we compute the gradients of $H_N$ at the points output by our algorithm, in both cases when $\vone$ is super-solvable and sub-solvable.
In particular, we show that they are approximate critical points on the product of spheres $\cS_N$ \revedit{(defined in \eqref{eq:def-SN})}.
As explained in Remark~\ref{rem:2-r-crits-relation}, in the strictly super-solvable case these $2^r$ outputs approximate the $2^r$ genuine critical points of $H_N$ on $\cS_N$.
The sub-solvable case of this computation is used in our companion paper \cite[Theorem 1.5(c) and Subsection 5.3]{huang2023strong} to show failure of \emph{annealed} topological trivialization in the sub-solvable case.
Finally in Subsection~\ref{subsec:branch} we give a modification of the tree-descending phase for the super-solvable case. It constructs $\exp(cN)$ well-separated approximate critical points arranged in a densely branching ultrametric tree; this implies the failure of \emph{strong} topological trivialization in \cite[Definition 6 and Theorem 1.6]{huang2023strong}.

\subsection{Notations}
\label{subsec:notation}

\revedit{Throughout, we will use boldface lowercase letters ($\bu,\bv,\ldots$) to denote vectors in $\bbR^N$, and lowercase letters with vector sign ($\vu,\vv,\ldots$) to denote vectors in $\bbR^\sS \simeq \bbR^r$.
Similarly, boldface uppercase letters denote matrices or tensors in $(\bbR^N)^{\otimes k}$, and non-boldface uppercase letters denote matrices or tensors in $(\bbR^r)^{\otimes k}$.}
We let 
\[
\< \bv\>_N=N^{-1}\sum_{i\le N} v_i;
\quad\quad
\< \bu,\bv\>_N
= 
N^{-1}\sum_{i\le N}u_iv_i
=
\langle\vec\lambda, \vR(\bu,\bv)\rangle
\]
for $\bu,\bv\in\reals^N$.
The corresponding norm is 
\[
\|\bu\|_{N}= \<\bu,\bu\>_N^{1/2}=\sqrt{\sum_s \lambda_s R_s(\bu,\bu)}.
\] 
Next $a_N\simeq b_N$ means that $a_N-b_N$ converges in probability to $0$. 
Analogously, for two vectors $\bu_N, \bv_N$, we write $\bu_N\simeq \bv_N$ when $\|\bu_N-\bv_N\|_N$ converges in probability to $0$. We denote limits in probability by $\plim_{N\to\infty}$.
Analogously we write $\approx_{\delta}$ to denote asymptotic equality as $\delta\to 0$.

For any tensor $\bA \in (\bbR^N)^{\otimes k}$, we define the operator norm
\[
    \tnorm{\bA}_{\op} = 
    \sup_{\|\bsig^1\|,\ldots,\|\bsig^k\| \leq 1} 
    \lt|\la \bA, \bsig^1 \otimes \cdots \otimes \bsig^k \ra\rt|.
\]
The following proposition shows that with exponentially good probability, the operator norms of all constant-order gradients of $H_N$ are bounded on the appropriate scale.
\begin{proposition}[{\cite[Proposition 1.13]{huang2023algorithmic}}]
\label{prop:gradients-bounded}
    For any fixed model $(\xi, \vh)$ there exists a constant $c>0$, sequence $(K_N)_{N\geq 1}$ of convex sets $K_N\subseteq \sH_N$, and sequence of constants $(C_{k})_{k\geq 1}$ independent of $N$, such that the following properties hold.
    \begin{enumerate}[label=(\alph*)]
        \item 
        \label{it:KN-high-prob}
        $\P[H_N\in K_N]\geq 1-e^{-cN}$;
        \item For all $H_N\in K_N$ and 
        $\bx\in \cB_N$,
        \begin{align}
            \label{eq:gradient-bounded}
            \norm{\nabla^k H_N(\bx)}_{\op}
            &\le 
            C_{k}N^{1-\frac{k}{2}}.
        \end{align}
    \end{enumerate}
\end{proposition}
\section{Achieving Energy $\ALG$}
\label{sec:amp}

In this section we prove Theorem~\ref{thm:main-alg} by exhibiting an approximate message passing (AMP) algorithm.
Throughout this section, Assumption~\ref{as:nondegenerate} on non-degeneracy of $\xi$ will be enforced without loss of generality. 

\subsection{\revedit{Definition of Pseudo-Maximizer}}

As mentioned before Definition~\ref{def:subsolvable-alg}, the threshold $\ALG$ in the sub-solvable case depends on a notion of pseudo-maximizer. 
We now provide this definition, which was derived in \cite[Theorem 3]{huang2023algorithmic} as a necessary condition for $\Phi$ to maximize $\bbA$ defined in \eqref{eq:alg-for-optimizer} (and it is proved therein that a maximizer always exists).

\begin{definition}    
    \label{def:pseudo-maximizer}
    A coordinate-wise strictly increasing $C^2$ function $\Phi:[q_1,1]\to [0,1]^{\sS}$, for some $q_1\in [0,1]$, is a pseudo-maximizer if:
    \begin{enumerate}[label=(\arabic*)]
        \item $\Phi$ is \emph{admissible}, meaning it satisfies the normalization 
        \begin{equation}
        \label{eq:admissible}
        \la \vlam, \Phi(q)\ra = q,\quad\forall q\in [q_1,1].
        \end{equation}
        In particular $\Phi(1) = \vone$.
        \item $\Phi(q_1)$ is solvable.
        \item The derivative at $q_1$ satisfies $M^*(\Phi(q_1))\Phi'(q_1)=\vzero$. This amounts to no restriction when $\vh=\vzero$ and thus $(q_1,\Phi(q_1))=(0,\vzero)$; when $\vh \neq \vzero$ it means that
        \begin{equation}
            \label{eq:Phip-q1}
            \Phi_s'(q_1)
            =
            \frac{\Phi_s(q_1) (\xi^s\circ\Phi)'(q_1)}{\xi^s(\Phi(q_1))+h_s^2},
            \quad
            s\in\sS
            .
        \end{equation}
        \item For all $q\in [q_1,1]$, $\Phi$ solves the (second-order) \emph{tree-descending} differential equation:
        \begin{equation}
            \label{eq:tree-descending-ode}
            \Psi(q)
            \equiv
            \fr{1}{\Phi'_s(q)}
        \deriv{q}
        \sqrt{\fr{\Phi'_s(q)}{(\xi^s \circ \Phi)'(q)}}
        \end{equation}
        is independent of the species $s$.
        (See \cite[Lemma 4.37]{huang2023algorithmic} for well-posedness of this ODE.)
    \end{enumerate}
\end{definition}

Note that there may exist multiple such $\Phi$, see \cite[Figure 2]{huang2023algorithmic}.
If $\vone$ is super-solvable, we adopt the convention that $q_1=1$ and $\Phi$ has domain $\{1\}$.

We now give an efficient approximate message passing algorithm achieving energy $\bbA(\Phi)$ for any pseudo-maximizer $\Phi$.
In particular for the optimal pseudo-maximizer this achieves energy $\ALG$.

\subsection{Review of Approximate Message Passing}
\label{subsec:amp}

Here we recall the class of approximate message passing algorithms, specialized to our setting of interest. We initialize AMP with a deterministic vector $\bw^0$ with coordinates
\begin{equation}
\label{eq:AMP-init-concrete}
    w^0_i = w_{s(i)}
\end{equation}
depending only on the species.
Let $f_{t,s}:\bbR^{t+1}\to\bbR$ be a Lipschitz function for each $(t,s)\in \bbZ_{\geq 0}\times \sS$. For $(\bw^0,\bw^1,\dots,\bw^t)\in\bbR^{N\times (t+1)}$, let 
$f_{t}(\bw^0,\bw^1,\dots,\bw^t)\in\bbR^N$ be given by
\[
    f_{t}(\bw^0,\bw^1,\dots,\bw^t)_i
    =
    f_{t,s(i)}(w^1_i,\dots,w^t_i),\quad i\in [N].
\]
We generate subsequent iterates through recursions of the following form, where $\ons_t$ is known as the \emph{Onsager correction term}:
\begin{align}
\label{eq:AMP-body}
    \bw^{t+1}
    &=
    \nabla H_N(\bm^t)
    -
    \ons_t
    ;
    \\
\nonumber
    \bm^t
    &=
    f_{t}(\bw^0,\bw^1,\dots,\bw^t);
    \\
\label{eq:ONS-body}
    \ons_t
    &=
    \sum_{t'\leq t}
    d_{t,t'}
    \diamond
    f_{t'-1}(\bw^1,\dots,\bw^{t'-1});
    \\
\label{eq:dts-def}
    d_{t,t',s}
    &=
    \lt(
    \sum_{s'\in\sS}
    \partial_{x_{s'}}
    \xi^s
    \lt(
    \big(
    \bbE[M^t_{s''} M^{t'-1}_{s''}]
    \big)_{s''\in\sS}
    \rt)
    \cdot
    \bbE
    \lt[
    \partial_{W^{t'}_{s'}}f_{t,s'}(W^0_{s'},\dots,W^t_{s'})
    \rt]
    \rt)
    .
\end{align}
Here $W^t_s,M^t_s$ are defined as follows. $W^0_s=w_s$ and the variables
$(\wt W^t_s)_{(t,s)\in \bbZ_{\geq 1}\times \sS}$ form a centered Gaussian process with covariance defined recursively by
\begin{equation}
\label{eq:state-evolution-basic}
\begin{aligned}
    \bbE[\wt W^{t+1}_s \wt W^{t'+1}_{s}]
    &=
    \xi^s\lt(\bbE[f_{t,s}(W^0_s,\dots,W^{t}_s)f_{t',s}(W^0_s,\dots,W^{t'}_s)]\rt),
    \\
    W^t_s
    &=
    \wt W^t_s+h_s;
    \\
    M^t_s
    &=
    f_{t,s}(W^0_s,\dots,W^t_s)
\end{aligned}
\end{equation}
and $\bbE[\wt W^{t+1}_{s} \wt W^{t'+1}_{s'}]=0$ if $s\neq s'$ (i.e. different species are independent).

The following \emph{state evolution} characterizes the behavior of the above iterates. 
It states that for each $s\in\sS$, when $i\in \cI_s$ is uniformly random the sequence of coordinates $(w^1_i,w^2_i,\dots,w^t_i)$ has the same law as $(W^1_s,\dots,W^t_s)$. 
Say a function $\psi:\bbR^{\ell} \to \bbR$ is pseudo-Lipschitz if $|\psi(x) - \psi(y)| \le C(1+|x|+|y|)|x-y|$ for a constant $C$.

\begin{proposition}
\label{prop:state_evolution}
For any pseudo-Lipschitz function $\psi$ and $\ell\in\bbZ_{\geq 0}$, $s\in\sS$,
\begin{equation}
\label{eq:SE-body}
    \plim_{N\to\infty}\frac{1}{N_s}
    \sum_{i\in\cI_s}
    \psi(\bw^0_i,\dots,\bw^{\ell}_i)
    =
    \bbE
    [
    \psi(W^0_s,\dots,W^{\ell}_s)
    ]
    .
\end{equation}
\end{proposition}

This proposition allows us to read off normalized inner products of the AMP iterates, since e.g.
\[
    \langle \bw^k,\bw^{\ell}\rangle_N
    \simeq
    \sum_{s\in\sS}
    \lambda_s
    \bbE[W^k_s W^{\ell}_s].
\]
Proposition~\ref{prop:state_evolution} is proved in Appendix~\ref{sec:ProofSE}. In fact we show a slight generalization allowing $f_t=f_t(\bw^0,\dots,\bw^t,\bg^0,\dots,\bg^t)$ to depend also on independently generated vectors $(\bg^0,\dots,\bg^t)\in\bbR^{N(t+1)}$. When using this extension, we will always take each $\bg^t\sim\cN(0,I_N)$ to be standard Gaussian. The more general result essentially says that $\bg_t$ still acts as an independent Gaussian for the purposes of state evolution. Since this is relatively intuitive, we refer to Theorem~\ref{thm:mixedAMP} in the appendix for a precise statement.

For random matrices (i.e. the case of quadratic $H$) there is a considerable literature establishing state evolution in many settings beginning with \cite{bolthausen2014iterative,BM-MPCS-2011} and later \cite{bayati2015universality,berthier2019state,chen2020universality,fan2020approximate,dudeja2022universality} (see also \cite{feng2022unifying} for a survey of many statistical applications). The generalization to tensors was introduced in \cite{richard2014statistical} and proved in \cite{ams20}, whose approach we follow.

\subsection{Stage $\I$: Finding the Root of the Ultrametric Tree}
\label{subsec:root-finding}

Our goal in this subsection will be to compute a vector $\bm^{\ul}$ satisfying
\[
    \plim_{\ul\to\infty}
    \lim_{N\to\infty}
    \vR(\bm^{\ul},\bm^{\ul})=\Phi(q_1)
\]  
and with the correct energy value (as stated in Lemma~\ref{lem:sphereenergy} below).
We take as given a maximizer $\Phi$ to $\bbA$ with domain $[q_1,1]$.
Recall $\Phi(q_1)$ is super-solvable: either $\vone$ is strictly sub-solvable, in which case $\Phi(q_1)$ is solvable, or $\vone$ is super-solvable, in which case $\Phi(q_1) = \Phi(1) = \vone$.

We use the initialization
\[
    w^0_i = \sqrt{\xi^s(\Phi(q_1))+h_s^2},\quad i\in \cI_s.
\]
Define the vector $\va\in\mathbb R^{\sS}$ by 
\[
    a_s=
    \sqrt{\frac{\Phi_s(q_1)}{\xi^s(\Phi(q_1))+h_s^2}}.
\]
Subsequent iterates are defined via the following recursion.
\begin{align}
\label{eq:RSsphere}
  \bw^{k+1}
  &=
  \nabla H_N(\bm^k)
  -
  \vb_k
  \diamond
  \bm^{k-1}
  \\
\nonumber
  &=
  \bh
  +
  \nabla \wtH_N(\bm^k)
  -
  \vb_k
  \diamond
  \bm^{k-1};
  \\
\label{eq:mk-def}
  \bm^k
  &=
    \va\diamond \bw^k
\\
\label{eq:zeta-defn}
    b_{k,s}
    &\equiv
    \sum_{s'\in\sS}
    a_{s'}
    \partial_{s'}\xi^s
    \big(\vR(\bm^k,\bm^{k-1})\big)
    .
\end{align}
The last term in \eqref{eq:RSsphere} comes from specializing the formula \eqref{eq:ONS-body} for the Onsager term.

Next recalling \eqref{eq:state-evolution-basic}, let $(W^j_s,M^j_s)_{j\geq 0,s\in\sS}$ be the state evolution limit of the coordinates of 
\[
    (\bw^{0},\bm^{0},\dots,\bw^k,\bm^k)
\]
as $N\to\infty$. Concretely, each $W^j_s$ is Gaussian with mean $h_s$ and 
\[
    M^{j}_s=\sqrt{\frac{\Phi_s(q_1)}{\xi^s(\Phi(q_1))+h_s^2}
    }
    \cdot 
    W^j_s,
    \quad
    j\geq 0,~
    s\in\sS.
\]
We next compute the covariance of the Gaussians $\wt W^j_s = W^j_s - h_s$. 
Define $\valpha : \bbR_{\ge 0}^\sS \to \bbR_{\ge 0}^\sS$ by
\begin{equation}
    \label{eq:overlap-recursion-AMP}
    \alpha_s(\vx) = 
    \lt(\xi^s(\vx)+h_s^2\rt)\lt(\frac{\Phi_s(q_1)}{\xi^s(\Phi(q_1))+h_s^2}\rt)\,.
\end{equation}
Define the (deterministic) $\bbR_{\geq 0}^{\sS}$-valued sequence $(\vR^0,\vR^1,\dots)$ of asymptotic overlaps recursively by $\vR^0=\vzero$ and $\vR^{k+1} = \valpha(\vR^k)$.

\begin{lemma}
\label{lem:RSconverge}
For integers $0\leq j<k$, the following equalities hold (the first in distribution):
\begin{align} 
\label{eq:id1.0}
    W^j_s&\stackrel{d}{=} h_s+Z\sqrt{\xi^s(\Phi(q_1))},\quad Z\sim \cN(0,1)\\
\label{eq:id2.0}
    \mathbb E[\wt W^j_s \wt W^k_s]&=\xi^s(\vR^j)\\
\label{eq:id3.0}
    \mathbb E[(M^j_s)^2]&=\Phi_s(q_1)\\
\label{eq:id4.0}
    \mathbb E[M^j_s M^k_s]&=R^{j+1}_s.
\end{align}
\end{lemma}

\begin{proof}
We proceed by induction on $j$, first showing \eqref{eq:id1.0} and \eqref{eq:id3.0} together. As a base case, \eqref{eq:id1.0} holds for $j=0$ by initialization. For the inductive step, assume first that \eqref{eq:id1.0} holds for $j$. Then by the definition \eqref{eq:mk-def},
\begin{align*}
  \mathbb E\lt[(M^j_s)^2\rt]
  &=
  \lt(\xi^s(\Phi(q_1))+h_s^2\rt)\cdot 
    a_s^2
  \\
  &=
  \lt(\xi^s(\Phi(q_1))+h_s^2\rt)\cdot \lt(\frac{\Phi_s(q_1)}{\xi^s(\Phi(q_1))+h_s^2}\rt)
  \\
  &=\Phi_s(q_1)
\end{align*}
so that \eqref{eq:id1.0} implies \eqref{eq:id3.0} for each $j\geq 0$. On the other hand, state evolution directly implies that if \eqref{eq:id3.0} holds for $j$ then \eqref{eq:id1.0} holds for $j+1$. This establishes \eqref{eq:id1.0} and \eqref{eq:id3.0} for all $j\geq 0$.

We similarly show \eqref{eq:id2.0} and \eqref{eq:id4.0} together by induction, beginning with \eqref{eq:id2.0}. When $j=0$ it is clear because $\wt W^k_s$ is mean zero and independent of $\wt W^0_s$. 
Just as above, it follows from state evolution that \eqref{eq:id2.0} for $(j,k)$ implies \eqref{eq:id4.0} for $(j,k)$ which in turn implies \eqref{eq:id2.0} for $(j+1,k+1)$. Hence induction on $j$ proves \eqref{eq:id2.0} and \eqref{eq:id4.0} for all $(j,k)$.
\end{proof}

The next lemma is crucial and uses super-solvability of $\Phi(q_1)$.

\begin{lemma}
\label{lem:Rj-to-Phiq1}
    The limit $\vR^\infty \equiv \lim_{j\to\infty} \vR^j$ exists and equals $\Phi(q_1)$.
\end{lemma}

\begin{proof}
    First we observe that $\valpha$ (recall \eqref{eq:overlap-recursion-AMP}) is coordinate-wise strictly increasing in the sense that if $0\preceq x\prec y$ then $\valpha(x)\prec \valpha(y)$. 
    Moreover $\valpha(\vzero)\succ 0$ (assuming $\vh\neq 0$, else the result is trivial) and $\valpha(\Phi(q_1))=\Phi(q_1)$. 
    Therefore $\vR^\infty$ exists, $\valpha(\vR^\infty)=\vR^\infty$, and
    \[
        \vzero\preceq \vR^\infty\preceq\Phi(q_1).
    \]
    It remains to show that the above forces $\vR^\infty=\Phi(q_1)$ to hold.
    
    Let $M\in\bbR^{\sS\times \sS}$ be the matrix with entries $M_{s,s'}=\deriv{t}\valpha_s(\Phi(q_1)+te_{s'})|_{t=0}$ for $e_{s'}$ a standard basis vector. 
    Then $M$ is the derivative matrix for $\valpha$ at $\Phi(q_1)$ in the sense that for any $\vu\in\bbR^{\sS}$,
    \[
        \deriv{t}\valpha(\Phi(q_1)+t\vu)|_{t=0}=M\vu.
    \]
    We easily calculate that
    \[
        M_{s,s'} = \fr{\Phi_s(q_1) \partial_{x_s,x_{s'}}\xi(\Phi(q_1))}{\partial_{x_s}\xi(\Phi(q_1)) + \lambda_s h_s^2}.
    \]    
    We claim that for any entry-wise non-negative vector $\vw\in\mathbb R_{\geq 0}^{\sS}$,
    \begin{equation}
    \label{eq:Mws} 
        (M\vw)_s\leq w_s
    \end{equation}
    for some $s\in \sS$. 
    Indeed, suppose to the contrary that $(M\vw)_s > w_s$ for all $s\in \sS$.
    This rearranges to
    \[
        \fr{\partial_{x_s}\xi(\Phi(q_1)) + \lambda_s h_s^2}{\Phi_s(q_1)} w_s 
        - \sum_{s'\in \sS} \partial_{x_s,x_{s'}} \xi(\Phi(q_1)) w_{s'} < 0
        \quad \forall s\in \sS,
    \]
    i.e. $M^*(\Phi(q_1)) \vw \prec \vzero$ (recall \eqref{eq:M*sym}).
    Proposition~\ref{prop:diagonally-signed-min-max} then implies that $\blambda_{\min}(M^*(\Phi(q_1))) < 0$, so $\Phi(q_1)$ is strictly sub-solvable, which is a contradiction. 
    Thus \eqref{eq:Mws} holds for some $s\in \sS$. 

    Now suppose for sake of contradiction that $\vR^\infty \prec \Phi(q_1)$, let $\vw=\Phi(q_1)-\vR^\infty$, and choose $s\in\sS$ such that \eqref{eq:Mws} holds. Write $f(t)=\alpha_s(\Phi(q_1)+t\vw)$. Since $\alpha_s$ is a polynomial with non-negative coefficients and $\xi$ is non-degenerate, $f$ is strictly convex and strictly increasing on $[-1,0]$. Hence
    \[
        \alpha_s(\vR^\infty)
        = f(-1)
        >
        f(0)-f'(0)
        \geq
        \Phi_s(q_1)-(M\vw)_s
        \stackrel{\eqref{eq:Mws}}{\geq}
        \Phi_s(q_1)-w_s
        =
        R^\infty_s.
    \]
    The first inequality above is strict, so we deduce that $\valpha(\vR^\infty)\neq\vR^\infty$ if $\vR^\infty\prec\Phi(q_1)$. This contradicts the definition of $\vR^\infty$. Therefore $\vR^\infty=\Phi(q_1)$, completing the proof.
\end{proof}

\begin{remark}
Super-solvability of $\Phi(q_1)$ is a tight condition for the above argument to hold, as the matrix $M$ above needs to have Perron-Frobenius eigenvalue at most $1$. Indeed suppose that $\Phi(q_1)$ was chosen so that $\lambda_1(M)>1$. Then there exists $\vw\in\bbR_{>0}^{\sS}$ with $M\vw\succ \vw$.
Letting $\vx=\Phi(q_1)-\eps \vw$ for small $\eps>0$, we find $\valpha(\vx)\prec \vx$.
Monotonicity implies that $\valpha$ maps the compact, convex set
\[
    K=\{\vy\in[0,1]^{\sS}~:~\vzero\preceq \vy\preceq \vx\}
\]
into itself. By the Brouwer fixed point theorem, a fixed point of $\valpha$ strictly smaller than $\Phi(q_1)$ exists whenever $\Phi(q_1)$ is strictly subsolvable.
\end{remark}

We finish our analysis of the first AMP phase by computing the asymptotic energy it achieves. As expected, the resulting value agrees with the first term in the formula \eqref{eq:alg-for-optimizer} for $\ALG$.
\begin{lemma}
\label{lem:sphereenergy}
\[
  \lim_{k\to\infty} \plim_{N\to\infty}\frac{H_N(\bm^k)}{N}
  = 
    \sum_{s\in \sS}
    \lambda_s
      \sqrt{
      \Phi_s(q_1)
      \cdot
      \lt(h_s^2+\xi^s(\Phi(q_1))\rt)
      }
    \,.
\]
\end{lemma}

\begin{proof}

We use the identity
\begin{equation}
  \frac{H_N(\bm^k)}{N}=\big\langle \bh,\bm^k\rangle_N+\int_0^1 \langle \bm^k,\nabla \widetilde H_N(t\bm^k)\big\rangle_N \de t
\end{equation}
and interchange the limit in probability with the integral. To compute $\plim_{N\to\infty}\langle \bm^k,\nabla \widetilde H_N(t\bm^k)\rangle$ we introduce an auxiliary AMP step 
\[
    \by^{k+1}=\nabla \widetilde H_N(t\bm^k)-
    t
    \vb_k
    \diamond
    \bm^{k-1}
\] 
which depends implicitly on $t\in [0,1]$.
Rearranging yields
\begin{align*}
  \vR(\bm^k,\nabla \widetilde H_N(t\bm^k)) 
  &= 
  \vR(\bm^k,\by^{k+1})
  +
  t\cdot
  \lt(
  \vR(\bm^k,\bm^{k-1} )
  \odot 
  \vb_k
  \rt)  
  \\
  &\simeq   
  \vR(\bm^k,\by^{k+1})
  +
  t\cdot
  \lt(
  \vR^k
  \odot 
  \vb_k
  \rt)  
  .
\end{align*}

For the first term, recalling \eqref{eq:mk-def} yields
\[
    R_s(\bm^k,\by^{k+1})
    =
    \mathbb E[a_s W^k_s Y^{k+1}_s]
    \\
    = 
    a_s
    \,
    \xi^s(t \vR^k)
    .
\]
Note also that 
\begin{equation}
    \label{eq:flip-partial}
    \lambda_s \partial_{s'}\xi^s(\vR^k)=\partial_{x_s,x_{s'}}\xi(\vR^k)=\lambda_{s'}\partial_{s}\xi^{s'}(\vR^k).
\end{equation}
Integrating with respect to $t$, and switching the roles of $s,s'$ in applying \eqref{eq:flip-partial} below, we thus find
\begin{align*} 
    \int_0^1 \langle \bm^k,\nabla \widetilde H_N(t\bm^k)\rangle_N \de t
    &\simeq 
    \sum_{s\in\sS}
    \lambda_s 
    \int_0^1
    R_s(\bm^k,\nabla \widetilde H_N(t\bm^k))
    \de t
    \\
    &\simeq
    \sum_{s\in\sS}
    \lambda_s
    \int_0^1
    \Big(
    a_s \xi^s(t\vR^k)
    +
    t R^k_s 
    \sum_{s'}
    a_{s'}\partial_{s'}\xi^s(\vR^k)
    \Big)
    ~\de t
    \\
    &\stackrel{\eqref{eq:flip-partial}}=
    \sum_{s\in\sS}
    \lambda_s
    \int_0^1
    \Big(
    a_s \xi^s(t\vR^k)
    +
    t 
    a_s  
    \sum_{s'}
    R^k_{s'}
    \partial_{s'}\xi^s(\vR^k)
    \Big)
    ~\de t
    \\
    &=
    \sum_{s\in\sS}
    \lambda_s
    a_s
    \int_0^1
    \frac{\de ~}{\de t}
    \lt(t\, \xi^s(t\, \vR^k)\rt)
    \de t
    \\
    &=
    \sum_{s\in\sS}
    \lambda_s a_s \xi^s(\vR^k).
\end{align*}
Finally the external field $\bh$ gives energy contribution
\[
  \langle \bh, \bm^k\rangle_N 
  \simeq 
  \sum_{s\in\sS}
  \lambda_s
  h_s\bbE[M^k_s]
  =
  \sum_{s\in\sS}
  \lambda_s
  a_s
  h_s^2.
\]
Since $\vR^\infty=\Phi(q_1)$ by Lemma~\ref{lem:Rj-to-Phiq1}, we conclude
\begin{align*}
  \lim_{k\to\infty} \plim_{N\to\infty}\frac{H_N(\bm^k)}{N}
  &= 
  \sum_{s\in\sS}
  \lambda_s
  a_s\big(h_s^2 + \xi^s(\Phi(q_1))\big)
  \\
  &=
  \sum_{s\in\sS}
  \lambda_s
  \sqrt{
    \Phi_s(q_1)
  \cdot
  \lt(h_s^2+\xi^s(\Phi(q_1))\rt)
  }
  .
  \qedhere
\end{align*}
\end{proof}

\subsection{Stage $\II$: Descending the Ultrametric Tree}

We now turn to the second phase which uses incremental approximate message passing. Choose a large integer $\ul$, and with $\delta=\ul^{-1}$ let
\[
    q^{\delta}_{\ell} 
    =
    q_1
    +
    (\ell-\ul)\delta,\quad \ell\geq 0.
\]
We then define
\begin{equation}
\label{eq:bn-def}
    \bn^{\ul}
    =
    \bm^{\ul}
    +
    \sqrt{\Phi(q_1+\delta)-\Phi(q_1)}\diamond \bg
\end{equation}
with the square-root taken entrywise, and $\bg\sim\cN(0,I_N)$. Then 
\begin{equation}
\label{eq:bn-overlap}
    \vR(\bn^{\ul},\bn^{\ul})\simeq \Phi(q_1+\delta)=\Phi(q^{\delta}_{\ul+1}).
\end{equation}
The point $\bn^{\ul}$ will be the ``root'' of our IAMP algorithm.\footnote{If $\vh=0$, one takes $\ul=q_1=0$, $n^{1}_i=\sqrt{\Phi_{s(i)}(\delta)}\bg_i$, and proceeds identically.}

Moreover we set $\ol=\max\{\ell\in\bbZ_+~:~q_{\ell}^{\delta}\leq 1-2\delta\}.$
We also define for $s\in\sS$ and $\ul\leq \ell\leq \ol$ the constants
\begin{equation}
\label{eq:u-def}
    u_{\ell,s}^{\delta}
    =
    \sqrt{
    \frac{\Phi_s(q^{\delta}_{\ell+1})-\Phi_s(q^{\delta}_{\ell})}
    {
    \xi^s(\Phi(q_{\ell+1}^{\delta}))
    -
    \xi^s(\Phi(q_{\ell}^{\delta}))
    }
    }
    .
\end{equation}

Set $\bz^{\ul}=\bw^{\ul}-\bh$.
We will define $(\bz^{\ell})_{\ell\geq \ul+1}$ via
\begin{equation}
\label{eq:general_amp}
\begin{aligned}
    \bz^{\ell+1} 
    &= 
    \nabla \widetilde H_N(f_{\ell}(\bz^{\ul},\cdots,\bz^\ell)) - \sum_{j=0}^\ell d_{\ell, j}\diamond f_{j-1}(\bz^{\ul},\cdots,\bz^{j-1})
    .
\end{aligned}
\end{equation}
The Onsager coefficients $d_{\ell,j}$ are given by \eqref{eq:dts-def} and will not appear explicitly in any calculations until Subsection~\ref{subsec:grad}. Note that formally, they may depend on the first $\ul$ iteratates, since \eqref{eq:general_amp} is a continuation of the same AMP iteration.
To complete the definition of the iteration \eqref{eq:general_amp}, for $s(i)=s$ and $\ell\geq \ul$ we set
\begin{equation}
\label{eq:n-def}
    f_{\ell,s}(z^{\ul}_i,\dots,z^{\ell}_i)
    =
    n^{\ell}_i,
\end{equation}
where
\begin{equation}
\label{eq:IAMP}
    \bn^{\ell+1} 
    =
    \bn^{\ell}+ 
    u_{\ell}^{\delta}
    \diamond
    \lt(\bz^{\ell+1}-\bz^{\ell}
    \rt).
\end{equation} 
The algorithm $\cA$ outputs
\begin{equation}
\label{eq:round-final-output}
    \cA(H_N)
    =
    \vR(\bn^{\ol},\bn^{\ol})^{-1/2}\diamond \bn^{\ol}
    \in\cB_N
\end{equation}
where the power $-1/2$ is taken entry-wise.
We show in \eqref{eq:final-overlap-alg} below that 
\[
    \lim_{\ul\to\infty}
    \plim_{N\to\infty}
    \|\bn^{\ol}-\cA(H_N)\|_N=0.
\]
Hence we will often not distinguish between the two and just consider $\bn^{\ol}$ to be the output. 
This makes essentially no difference by virtue of Proposition~\ref{prop:gradients-bounded}.

The state evolution limits of $\bz^\ell$ and $\bn^\ell$ are described by time-changed Brownian motions with total variance $\Phi_s(q^{\delta}_{\ell})$ in species $s$ after iteration $\ell$. This is made precise below.
\begin{lemma}
\label{lem:BMlimit}
Fix $s\in\sS$. The sequences $(Z^{\delta}_{\ul,s},Z^{\delta}_{\ul+1,s},\dots)$ and $(N^{\delta}_{\ul,s},N^{\delta}_{\ul+1,s},\dots)$ are Gaussian processes satisfying
\begin{align} 
\label{eq:BM1}
    \mathbb E[(Z^{\delta}_{\ell+1,s}-Z^{\delta}_{\ell,s})Z^{\delta}_{j,s}]
    &=
    0,\quad \text{for all }\ul+1\leq j\leq \ell
    \\
\label{eq:BM2}
    \mathbb E\big[
    (Z^{\delta}_{\ell+1,s}-Z^{\delta}_{\ell,s})^2
    \big]
    &=
    \xi^s(\Phi(q_{\ell+1}^{\delta}))
    -
    \xi^s(\Phi(q_{\ell}^{\delta}))
    \\
\label{eq:BM3}
    \mathbb E[Z^{\delta}_{\ell,s}Z^{\delta}_{j,s}]
    &=
    \xi^s(\Phi(q_{j\wedge \ell}^{\delta}))
    \\
\label{eq:BM4}
    \mathbb E[N^{\delta}_{\ell,s}N^{\delta}_{j,s}]
    &=
    \Phi_s(q^{\delta}_{(j\wedge \ell)+1})
    .
\end{align}
\end{lemma}

\begin{proof}
The fact that these sequences are Gaussian processes is a general fact about state evolution (the external Gaussian $\bg$ is permitted in Theorem~\ref{thm:mixedAMP}). 
We proceed by induction on $\ell\geq \ul$. The proof is similar to \cite[Section 8]{sellke2021optimizing} so we give only the main points (in fact \eqref{eq:bn-def} simplifies the corresponding construction therein, which avoided the use of external Gaussian noise). We will make liberal use of \eqref{eq:state-evolution-basic} to connect asymptotic overlaps before and after applying $\nabla H_N(\cdot)$.

For base cases, the $\ul$ case of \eqref{eq:BM3} is immediate from \eqref{eq:id3.0}.
The base case of \eqref{eq:BM4} follows from \eqref{eq:bn-overlap}, and thus the $\ul+1$ case of \eqref{eq:BM3}.
The main computation for the base case is
\begin{align*}
    \bbE\big[\big(Z^{\delta}_{\ul+1,s}-Z^{\delta}_{\ul,s}\big)Z^{\delta}_{\ul,s}\big]
    &=
    \xi^s\lt(\{\bbE[N^{\delta}_{\ul,s}M^{\ul-1}_s]\}_{s\in\sS}\rt)
    -
    \xi^s\lt(\{\bbE[M^{\ul-1}_{s}M^{\ul-1}_s]\}_{s\in\sS}\rt)
    \\
    &=
    \xi^s(\Phi(q_1))-\xi^s(\Phi(q_1))
    \\
    &=
    0.
\end{align*}
Here we used the general AMP statement of Theorem~\ref{thm:mixedAMP} to say that 
\[
    \bbE[N^{\delta}_{\ul,s} M^{\ul-1}_s]
    =
    \bbE[M^{\ul-1}_s M^{\ul-1}_s]
    =
    \Phi_s(q_1).
\]

For inductive steps, we always have by state evolution
\[
    \bbE[Z^{\delta}_{\ell+1,s}Z^{\delta}_{j+1,s}]
    \simeq
    \xi^s\big(\vR(\bn^{\ell},\bn^{j})\big).
\]
It follows by the inductive hypothesis of \eqref{eq:BM1} that for $j\leq \ell$,
\begin{align*}
    R_s(\bn^{\ell},\bn^{j})
    &=
    R_s(\bn^{\ul},\bn^{\ul})
    +
    \sum_{k=\ul}^{j-1}
    (u_k^{\delta})^2
    R_s(\bz^{k+1}-\bz^k,\bz^{k+1}-\bz^k)
    \\
    &=
     R_s(\bn^{\ul},\bn^{\ul})
     +
     \sum_{k=\ul}^{j-1}
     (u_k^{\delta})^2
     \lt(
      \xi^s(\Phi(q_{k+1}^{\delta}))
    -
    \xi^s(\Phi(q_{k}^{\delta}))
    \rt)
    \\
    &=
    \Phi_s(q_1)
    +
    \sum_{k=\ul}^{j-1}
    \Big(
    \Phi_s(q^{\delta}_{k+1})
    -
    \Phi_s(q^{\delta}_{k})
    \Big)
    \\
    &=
    \Phi_s(q^{\delta}_j).
\end{align*}
Plugging into the above yields that for $j\leq \ell$,
\[
    \bbE[Z^{\delta}_{\ell+1,s}Z^{\delta}_{j+1,s}]
    =
    \xi^s(\Phi(q^{\delta}_j)).
\]  
This depends only on $\min(j,\ell)$, so \eqref{eq:BM1} follows. The others are proved by similar computations.
\end{proof}

Equation~\eqref{eq:BM4} implies that $\vR(\bn^{\delta}_{\ell},\bn^{\delta}_{j})\simeq \Phi(q^{\delta}_{(\ell\wedge j)+1})$, which exactly corresponds to the previous sections of the paper. In particular it implies that the final iterate $\bn^{\delta}_{\ol}$ satisfies
\begin{equation}
\label{eq:final-overlap-alg}
    (1-O(\delta))\cdot\vone\preceq \vR(\bn^{\delta}_{\ol},\bn^{\delta}_{\ol})\preceq \vone
\end{equation}
so the rounding step \eqref{eq:round-final-output} causes only an $O(\delta)$ change in the Hamiltonian value. Finally we compute in Lemma~\ref{lem:iampenergy} below the energy gain from the second phase, which matches the second term in \eqref{eq:alg-for-optimizer}. 

\begin{lemma}
\label{lem:iampenergy}
\begin{equation}
    \label{eq:iampenergy}
    \lim_{\ul\to\infty}
    \plim_{N \to \infty}
    \frac{H_{N}(\bn^{\ol})-H_{N}\lt(\bn^{\ul}\rt)}{N} 
    = 
    \sum_{s\in\sS}
  \lambda_s
  \int_{q_1}^{1} 
    \sqrt{\Phi'_s(t) (\xi^s \circ \Phi)'(t)}
    \,
    \de t
\end{equation}
\end{lemma}

\begin{proof}
Observe that $\langle h,\bn^{\ol}-\bn^{\ul}\rangle_N\simeq 0$ because the values $(N_{\ell,s}^{\delta})_{\ell\geq\ul}$ form a martingale sequence for each $s\in\sS$. 
Therefore it suffices to find the in-probability limit of $\frac{\widetilde{H}_{N}(\bn^{\ol})-\widetilde{H}_{N}(\bn^{\ul})}{N}$. We write
\[
    \frac{\widetilde{H}_{N}(\bn^{\ol})-\widetilde{H}_{N}(\bn^{\ul})}{N}=\sum_{\ell=\ul}^{\ol-1}\frac{\widetilde{H}_{N}(\bn^{\ell+1})-\widetilde{H}_{N}(\bn^{\ell})}{N}
\]
and use a Taylor series approximation for each term. In particular for $F\in C^3(\mathbb R;\bbR)$, applying Taylor's approximation theorem twice yields
\begin{align*}
    F(1)-F(0)
    &=
    F'(0)+\frac{1}{2}F''(0)+O\big(\sup_{a\in [0,1]}|F'''(a)|\big)
    \\
    &= 
    F'(0)+\frac{1}{2}(F'(1)-F'(0))+O\big(\sup_{a\in [0,1]}|F'''(a)|\big)
    \\
    &=
    \frac{1}{2}(F'(1)+F'(0))+O\big(\sup_{a\in [0,1]}|F'''(a)|\big) .
\end{align*}
Assuming $\sup_{\ell} \norm{\bn^\ell}_N \leq 1$, which holds with probability $1-o_N(1)$ by state evolution and the definition of $\ol$, we apply this estimate with 
\[
    F(a)= \fr1N \widetilde{H}_N\lt((1-a)\bn^{\ell}+a\bn^{\ell+1}\rt).
\]
The result is:
\begin{align*}
    \fr1N
    \lt|
    \widetilde{H}_{N}
    (\bn^{\ell+1})-\widetilde{H}_{N}(\bn^{\ell}) -\frac{1}{2}\lt\langle \nabla \widetilde{H}_N(\bn^{\ell})+\nabla \widetilde{H}_N(\bn^{\ell+1}),\bn^{\ell+1}-\bn^{\ell}\rt\rangle \rt|
    &\leq 
    O\lt(
    \underline{C}
    \|\bn^{\ell+1}-\bn^{\ell}\|_N^3
    \rt)
    ;
    \\
    \underline{C}N^{-1/2}
    &=
    \sup_{\|\bsig\|\leq \sqrt{N}}\lt\|\nabla^3 \widetilde{H}_N(\bsig)\rt\|_{\op}
    .
\end{align*}
Proposition~\ref{prop:gradients-bounded} implies that for deterministic constants $c,C$,
\[
    \bbP[\underline{C}\leq C]\geq 1-e^{-cN}.
\]
On the other hand for each $\ul\leq \ell\leq \ol-1$ we have
\begin{align*}
    \plim_{N\to\infty}\|\bn^{\ell+1}-\bn^{\ell}\|_N &=
    \sqrt{
    \sum_{s\in\sS}\lambda_s R_s(\bn^{\ell+1}-\bn^{\ell},\bn^{\ell+1}-\bn^{\ell})
    }
    \\
    &=
    \sqrt{
    \sum_{s\in\sS}\lambda_s
    \big(
    \Phi_s(q^{\delta}_{\ell+2})-\Phi_s(q^{\delta}_{\ell+1})
    \big)
    }
    \\
    &=
    \sqrt{\delta}.
\end{align*}
Summing and noting that $\ol-\ul\leq \delta^{-1}$ yields the high-probability estimate
\begin{align*}
  \sum_{\ell=\ul}^{\ol-1}
  &
  \fr1N 
  \lt|
  \widetilde{H}_{N}(\bn^{\ell+1})
  -
  \widetilde{H}_{N}(\bn^{\ell}) 
  -
  \frac{1}{2}\lt\langle  \nabla \widetilde{H}_N(\bn^{\ell})+\nabla \widetilde{H}_N(\bn^{\ell+1}),\bn^{\ell+1}-\bn^{\ell}\rt\rangle  
  \rt|
  \\ 
  &\leq 
  \sum_{\ell=\ul}^{\ol-1} 
  \|\bn^{\ell+1}-\bn^{\ell}\|_N^3
  \leq O(\sqrt{\delta}).
\end{align*}
So, this term vanishes as $\delta \to 0$. It remains to prove
\[
  \lim_{\delta\to 0}
  \plim_{N \to \infty}
  \sum_{\ell=\ul}^{\ol-1}
  \lt\langle 
  \nabla \widetilde{H}_N(\bn^{\ell})
  +
  \nabla \widetilde{H}_N(\bn^{\ell+1})
  ,
  \bn^{\ell+1}-\bn^{\ell}
  \rt\rangle_N  
  \stackrel{?}{=}
  2
  \sum_{s\in\sS}
  \lambda_s
  \int_{q_1}^{1} 
    \sqrt{\Phi'_s(t) (\xi^s \circ \Phi)'(t)}
  \de t.
\]
To establish this it suffices to show for each species $s\in\sS$ the equality
\begin{equation}
\label{eq:IAMP-energy-per-species}
  \lim_{\delta\to0}
  \plim_{N \to \infty}
  \sum_{\ell=\ul}^{\ol-1}
  R_s\lt(
  \nabla \widetilde{H}_N(\bn^{\ell})
  +
  \nabla \widetilde{H}_N(\bn^{\ell+1})
  ,
  \bn^{\ell+1}-\bn^{\ell}
  \rt)
  \stackrel{?}{=}
  2
  \int_{q_1}^{1} 
    \sqrt{\Phi'_s(t) (\xi^s \circ \Phi)'(t)}
  \de t.
\end{equation}
Observe by \eqref{eq:general_amp} that
\begin{equation}
\label{eq:amprearrange}
    \nabla \widetilde{H}_N(\bn^{\ell})
    =
    \bz^{\ell+1}+\sum_{j=0}^\ell d_{\ell, j}\diamond \bn^{j-1}.
\end{equation}
Passing to the limiting Gaussian process $(Z^{\delta}_k)_{k\in\mathbb Z^+}$ via state evolution,
\begin{align*}
  \plim_{N\to\infty}
  R\lt(\nabla \widetilde{H}_N(\bn^{\ell}),\bn^{\ell+1}-\bn^{\ell}\rt)_s
  &=
  \mathbb E\lt[ Z^{\delta}_{\ell+1,s}(N^{\delta}_{\ell+1,s}-N^{\delta}_{\ell,s})\rt]
  +
  \sum_{j=0}^{\ell}
  d_{\ell,j,s}
  \mathbb E\lt[
  N^{\delta}_{j-1,s}(N^{\delta}_{\ell+1,s}-N^{\delta}_{\ell,s})
  \rt],
  \\
  \plim_{N\to\infty}
  R\lt(\nabla \widetilde{H}_N(\bn^{\ell+1}),\bn^{\ell+1}-\bn^{\ell}\rt)_s
  &=
  \mathbb E\lt[ Z^{\delta}_{\ell+2,s}(N^{\delta}_{\ell+1,s}-N^{\delta}_{\ell,s})\rt]
  +
  \sum_{j=0}^{\ell+1} 
  d_{\ell+1,j,s}
  \mathbb E\lt[
  N^{\delta}_{j-1}(N^{\delta}_{\ell+1,s}-N^{\delta}_{\ell,s})
  \rt].
\end{align*}

As $(N^{\delta}_k)_{k\geq \mathbb Z^+}$ is a martingale process, it follows that all right-most expectations vanish. Similarly it holds that
\begin{align*}
  \mathbb E[Z_{\ell+2}^{\delta}(N_{\ell+1}^{\delta}-N_{\ell}^{\delta})]&=\mathbb E[Z_{\ell+1}^{\delta}(N_{\ell+1}^{\delta}-N_{\ell}^{\delta})]
  \\
  \mathbb E[Z_{\ell}^{\delta}(N_{\ell+1}^{\delta}-N_{\ell}^{\delta})]&=0.
\end{align*}
We conclude that
\begin{align*}
  \plim_{N\to\infty}
  R\lt(\nabla \widetilde{H}_N(\bn^{\ell})+\nabla \widetilde{H}_N(\bn^{\ell+1}),\bn^{\ell+1}-\bn^{\ell}\rt)_s
  &=
  2\,\mathbb E[
  (Z_{\ell+1,s}^{\delta}-Z^{\delta}_{\ell,s})(N_{\ell+1,s}^{\delta}-N_{\ell,s}^{\delta})
  ]
  \\
  &=
  2\,\mathbb E[u_{\ell,s}^{\delta}(Z^{\delta}_{\ell,s})(Z_{\ell+1,s}^{\delta}-Z^{\delta}_{\ell,s})^2]
  \\
  &=
  2\,\mathbb E[u_{\ell,s}^{\delta}(Z^{\delta}_{\ell,s})]
  \cdot
  \bbE[(Z_{\ell+1,s}^{\delta}-Z^{\delta}_{\ell,s})^2]
  \\
&=
  2\,\sqrt{
  \Big(
  \Phi_s(q^{\delta}_{\ell+1,s})-\Phi_s(q^{\delta}_{\ell,s})
  \Big)
  \cdot
    \Big(\xi^s(\Phi(q_{\ell+1}^{\delta}))-\xi^s(\Phi(q_{\ell}^{\delta}))
    \Big)
    }
    .
\end{align*}
In the second-to-last step we used independence of $Z^{\delta}_{\ell,s}$ increments, which follows from Lemma~\ref{lem:BMlimit}, while the last step used \eqref{eq:u-def} and \eqref{eq:BM2}.
Combining with \cite[Lemma 3.7]{huang2023algorithmic} on discrete approximation of the integral in $\bbA$ implies \eqref{eq:IAMP-energy-per-species}.
\end{proof}

\begin{proof}[Proof of Theorem~\ref{thm:main-alg}]
We take $\cA$ as in \eqref{eq:round-final-output} for $\ul$ a large constant depending on $(\eps,\xi,h,\lambda)$. First,
\begin{equation}
\label{eq:1-o1-prob-alg}
    \bbP[H_N(\cA(H_N))/N \geq \ALG-\eps/2]
    \ge 
    1-o_N(1)
\end{equation}
follows from combining Lemma~\ref{lem:sphereenergy}, Lemma~\ref{lem:iampenergy} and the fact that (recall \eqref{eq:final-overlap-alg}) 
\[
H_N(\cA(H_N))/N \simeq H_N(\bn^{\ul})/N+o_{\bbP}(1).
\]

Next, let $K_N\subseteq\sH_N$ be as in Proposition~\ref{prop:gradients-bounded}. We recall that $\bbP[H_N\in K_N]\geq 1-e^{-cN}$.
Exactly as in \cite[Theorem 10]{huang2021tight} it follows that there is a $C(\eps)$-Lipschitz function $\wt\cA:\sH_N\to\bbR$ such that $\wt\cA$ and $\cA$ agree on $K_N$.
Moreover \eqref{prop:gradients-bounded} and concentration of measure on Gaussian space imply that $H_N(\wt\cA(H_N))$ is $O(N^{1/2})$-sub-Gaussian. In light of \eqref{eq:1-o1-prob-alg} and since $\bbP[\wt\cA(H_N)=\cA(H_N)]\geq \bbP[H_N\in K_N]\geq 1-e^{-cN}$, we deduce that 
\[
    \bbP[H_N(\cA(H_N))/N \geq \ALG-\eps]
    \ge 
    1-e^{-cN}.
\]
This concludes the proof.
\end{proof}

\section{Extensions}
\label{sec:branch}

\subsection{Signed AMP}
\label{subsec:signed}

\revedit{In our companion paper \cite{huang2023strong}, we show that strictly super-solvable models have w.h.p. exactly $2^r$ critical points, indexed by sign patterns $\vDelta \in \{\pm 1\}^r$ with the following physical meaning. 
Consider first the extreme case of a linear Hamiltonian, with external field $\bh = \vh \diamond \bone$ where all entries of $\vh$ are nonzero and \textbf{no other interactions}.
This model clearly has $2^r$ critical points, which are the products of the maxima and minima in the spheres $\{\norm{\bx_s}_2^2 = \lambda_s N\}$ corresponding to each species $s\in \sS$, and the signs $\vDelta$ record whether the critical point is a maximum or minimum in each species.
As explained in \cite[Section 6.6]{huang2023strong}, if a strictly super-solvable $H_N$ is gradually deformed to a linear function (staying inside the strictly super-solvable phase), the critical points move stably, and over this process their Hessian eigenvalues do not cross zero.
Thus, each critical point of $H_N$ can also be associated with a sign pattern $\vDelta$.}

\revedit{We now show that the root-finding algorithm defined in Subsection~\ref{subsec:root-finding} can be generalized to find all $2^r$ critical points in a strictly super-solvable model. More precisely, it finds $2^r$ approximate critical points, one in a neighborhood of each exact critical point of the model, from which the exact critical points can be computed by Newton's method (see Remark~\ref{rmk:supersolvable-unique-crit-nearby}). 
For general models, it finds $2^r$ approximate critical points on the product of spheres with self-overlap $\Phi(q_1)$. The restriction of $H_N$ to this set, considered as a spin glass in its own right (see \cite[Remark 1.2]{huang2023algorithmic}) is a solvable model.}

Fixing $\vDelta \in \{\pm 1\}^r$, the analogous iteration to \eqref{eq:RSsphere} is:
\begin{equation}
\label{eq:signed-AMP}
\begin{aligned}
    \bw^{k+1}
    &=
    \nabla H_N(\bm^k)
    -
    \vb_k
    \diamond
    \bm^{k-1}
    \\
    &=
    \bh
    +
    \nabla \wtH_N(\bm^k)
    -
    \vb_k(\vDelta)
    \diamond
    \bm^{k-1};
    \\
    \bm^k
    &=
    \vDelta\odot \va\diamond \bw^k
    \\
    b_{k,s}(\vDelta)
    &\equiv
    \sum_{s'\in\sS}
    \Delta_{s'}
    a_{s'}
    \partial_{s'}\xi^s
    \big(\vR(\bm^k,\bm^{k-1})\big)
    .
\end{aligned}
\end{equation}
The change of sign does not affect the proofs or statements of Lemmas~\ref{lem:RSconverge},~\ref{lem:Rj-to-Phiq1}. Indeed $a_s^2$ only changes to $\Delta_s^2 a_s^2$ in the former proof which is no change at all. The generalization of Lemma~\ref{lem:sphereenergy} is as follows.
\begin{lemma}
\label{lem:sphereenergy-signed}
\[
  \lim_{k\to\infty} \plim_{N\to\infty}\frac{H_N(\bm^k)}{N}
  = 
    \sum_{s\in \sS}
    \lambda_s
    \Delta_s
      \sqrt{
      \Phi_s(q_1)
      \lt(h_s^2+\xi^s(\Phi(q_1))\rt)
      }
    \,.
\]
\end{lemma}

\begin{proof}
The proof is similar to Lemma~\ref{lem:sphereenergy}. The main calculation now becomes:
\begin{align*} 
    \int_0^1 \langle \bm^k,\nabla \widetilde H_N(t\bm^k)\rangle_N \de t
    &\simeq 
    \sum_{s\in\sS}
    \lambda_s 
    \int_0^1
    R_s(\bm^k,\nabla \widetilde H_N(t\bm^k))
    \de t
    \\
    &\simeq
    \sum_{s\in\sS}
    \lambda_s
    \int_0^1
    \Big(
    \Delta_s a_s \xi^s(t\vR^k)
    +
    t R^k_s 
    \sum_{s'\in\sS}
    \Delta_{s'} 
    a_{s'}\partial_{s'}\xi^s(\vR^k)
    \Big)
    ~\de t
    \\
    &\stackrel{\eqref{eq:flip-partial}}=
    \sum_{s\in\sS}
    \lambda_s
    \int_0^1
    \Big(
    \Delta_s 
    a_s \xi^s(t\vR^k)
    +
    t 
    \Delta_s a_s  
    \sum_{s'\in\sS}
    R^k_{s'}
    \partial_{s'}\xi^s(\vR^k)
    \Big)
    ~\de t
    \\
    &=
    \sum_{s\in\sS}
    \lambda_s
    \Delta_s
    a_s
    \int_0^1
    \frac{\de ~}{\de t}
    \lt(t\, \xi^s(t\, \vR^k)\rt)
    \de t
    \\
    &=
    \sum_{s\in\sS}
    \lambda_s
    \Delta_s 
    a_s \xi^s(\vR^k).
\end{align*}
Moreover the external field $\bh$ now contributes energy
\[
  \langle \bh, \bm^k\rangle_N 
  \simeq 
  \sum_{s\in\sS}
  \lambda_s
  h_s\bbE[M^k_s]
  =
  \sum_{s\in\sS}
  \lambda_s
  \Delta_s 
  a_s
  h_s^2.
\]
Combining gives the desired statement.
\end{proof}

\begin{remark}
\label{rem:2-r-crits-relation}
    One can sign the IAMP phase as well by redefining \eqref{eq:IAMP} to
    \begin{equation}
    \label{eq:IAMP-signed}
    \bn^{\ell+1}(\vDelta)
    =
    \bn^{\ell}(\vDelta)
    + 
    \vDelta
    \odot
    u_{\ell}^{\delta}
    \diamond
    (\bz^{\ell+1}-\bz^{\ell}).
    \end{equation}
    The resulting output $\bn^{\ol}(\vDelta)$ then achieves asymptotic energy (recall \eqref{eq:alg-for-optimizer})
    \begin{equation}
    \label{eq:iamp-signed-energy}
    \lim_{\ul\to\infty}
    \plim_{N \to \infty}
    \frac{H_N
    \big(\bn^{\ol}(\vDelta)\big)}
    {N} 
    =
    \sum_{s\in \sS}
    \lambda_s \Delta_s\lt[
        \sqrt{\Phi_s(q_1) (\xi^s(\Phi(q_1)) + h_s^2)}  + 
        \int_{q_1}^1 \sqrt{\Phi'_s(q)(\xi^s\circ \Phi)'(q)}~\de q
    \rt].
    \end{equation}
    However it is unclear whether $\bn^{\ol}(\vDelta)$ can be made to obey any notable properties. 
    We will show that the signed outputs $\bm^k(\vDelta)$ of the first phase above are approximate critical points for $H_N$ (and in \cite{huang2023strong} that all near-critical points are close to one of them). 
    By contrast, for the output of signed IAMP to be a critical point, $\Phi$ must satisfy a signed version of the tree-descending ODE \eqref{eq:tree-descending-ode} in which the function $(\xi^s \circ \Phi)'(q)$ is replaced by
    \[
    \sum_{s'\in\sS}
    \Delta_{s'}
    \partial_{s'}\xi^s(\Phi(q))
    \Phi_s'(q).
    \]
    Since this quantity appears inside a square root in \eqref{eq:tree-descending-ode}, it is unclear when to expect solutions to exist. Furthermore the proof in \cite{huang2023algorithmic} of well-posedness relies on positivity of coefficients (via Perron-Frobenius theory) and does not seem to generalize. 
    \revedit{Additionally, a solution would not seem to correspond to a maximizer of any variational problem as in \eqref{eq:alg-for-optimizer}. As a result we do not know how to prove a solution exists in the signed case.}
    However if one takes \emph{as given} a smooth function $\Phi$ satisfying the signed tree-descending ODE, the iteration \eqref{eq:IAMP-signed} starting from signed initialization $\bn^{\ul}(\vDelta)=\bm^{\ul}(\vDelta)+
    \sqrt{\Phi(q_1+\delta)-\Phi(q_1)}\diamond \bg$ would produce an approximate critical point $\bn^{\ol}(\vDelta)$ which still satisfies \eqref{eq:iamp-signed-energy}.
\end{remark}

\subsection{Gradient Computation and Connection to $E_{\infty}$}
\label{subsec:grad}

We now compute the gradient of the outputs, showing that $\bm^{\ul}(\vDelta)$ and $\bn^{\ell}$ ($\ul \le \ell \le \ol$) are approximate critical points for the restriction of $H_N$ to the products of $r$ spheres with suitable radii passing through them.
For $\bsig$ to be an approximate critical point means precisely that there exist coefficients $\vA \in \bbR^r$ such that
\[
    \|\nabla H_N(\bsig)
    -
    \vA\diamond \bsig\|_N
    \simeq 0.
\]
In our case, these coefficients will be given as follows.
If $\vone$ is strictly sub-solvable (so $q_1<1$), define $\vA(q)$ for $q\in [q_1,1]$ by
\begin{align}
    \label{eq:as-def}
    A_s(q)
    &\equiv
    f_s(q)^{-1}
    +\sum_{s'\in\sS}
    f_{s'}(q)
    \partial_{s'}\xi^s\big(\Phi(q)\big),
    \\
    \label{eq:f-def}
    f_s(q)
    &\equiv
    \sqrt{\fr{\Phi'_s(q)}{(\xi^s \circ \Phi)'(q)}}.
\end{align}  
Further define for $\vDelta \in \{-1,1\}^r$
\begin{equation}
\label{eq:as-alt-formula}
    A_s(q_1;\vDelta)
    \equiv
    \Delta_s
    \sqrt{\frac{\xi^s(\Phi(q_1))+h_s^2}{\Phi_s(q_1)}}
    +
    \sum_{s'\in\sS}
    \Delta_{s'}
    \partial_{s'}\xi^s\big(\Phi(q_1)\big)
    \sqrt{
    \frac{\Phi_{s'}(q_1)}
    {\xi^{s'}(\Phi(q_1))+h_{s'}^2}
    }.
\end{equation}
Note that, by \eqref{eq:Phip-q1}, this is consistent with the definition of $\vA(q_1)$ above, in the sense that $\vA(q_1;\vone) = \vA(q_1)$.
We take this to be the definition of $\vA(q_1)$ if $\vone$ is super-solvable (and $q_1=1$).
\begin{proposition}
\label{prop:approx-crit-1}
    If $\Phi$ is a pseudo-maximizer for $\bbA$ \revedit{(recall Definition~\ref{def:pseudo-maximizer})} then for any $\vDelta\in\{\pm 1\}^r$,
    \begin{equation}
    \label{eq:approx-crit-1}
    \lim_{\ul\to\infty}
    \plim_{N\to\infty}
    \|\nabla H_N(\bm^{\ul}(\vDelta))-
    \vA(q_1,\vDelta) 
    \diamond
    \bm^{\ul}(\vDelta)\|_N
    =
    0
    .
    \end{equation}
\end{proposition}

\begin{proof}
    Recall from Lemma~\ref{lem:RSconverge} (which holds without modification for general $\vDelta$) that
    \begin{equation}
    \label{eq:recall-amp-convergence}
    \lim_{\ul\to\infty}\plim_{N\to\infty} 
    \|\bm^{\ul+1}(\vDelta)-\bm^{\ul}(\vDelta)\|_N 
    =0.
    \end{equation}
    Thus rearranging \eqref{eq:signed-AMP} yields 
    \[
    \lim_{k\to\infty}\plim_{N\to\infty} 
    \|\nabla H_N(\bm^{\ul}(\vDelta))
    -
    (\vDelta\odot\va^{-1}
    +
    \vb_k(\vDelta))
    \diamond
    \bm^{\ul}(\vDelta)
    \|_N
    =0.
    \]
    Since $\lim_{\ul\to\infty} 
    \big(\Delta_s a_s^{-1}+b_{\ul,s}(\vDelta)\big)=A_s(\vDelta)$ by \eqref{eq:as-alt-formula},
    the result follows.
\end{proof}

\begin{remark}
    \label{rmk:supersolvable-unique-crit-nearby}
    In \cite[Theorems 1.5 and 1.6]{huang2023strong}, we show that when $\xi$ is strictly super-solvable, $H_N$ has exactly $2^r$ critical points $\{\bx(\vDelta)\}_{\vDelta\in \{-1,1\}^r}$.
    Moreover all $\eps$-approximate critical points with Riemannian gradient $\|\nabla_{\sph}H_N(\bx)\|\leq\eps\sqrt{N}$ are within $o_{\eps}(\sqrt{N})$ of some $\bx(\vDelta)$.
    It follows from Proposition~\ref{prop:approx-crit-1} that each $\bm^{\ul}(\vDelta)$ is an $\eps$-approximate critical point for large enough $\ul=\ul(\xi,\eps)$. 
    In fact the preceding gradient computation shows that the values $\vDelta$ agree, implying that $\|\bm^{\ul}(\vDelta)-\bx(\vDelta)\|_N \le o_{\ul\to\infty}(1)$ (compare with \cite[Definition 5, Eq. (1.15)]{huang2023strong}). 
    Moreover by \cite[Theorem 1.6]{huang2023strong} each Riemannian Hessian $\nabla^2_{\sph}H_N(\bx(\vDelta))$ has condition number at least $1/C(\xi)$. 
    It follows that each critical point $\bx(\vDelta)$ can be efficiently computed to arbitrary accuracy by applying Newton's method from $\bm^{\ul}(\vDelta)$ for a large enough $\ul=\ul(\xi)$. (By contrast, the convergence of $\bm^{\ul}(\vDelta)$ itself to $\bx(\vDelta)$ is only in the careful double-limit sense $\lim_{\ul\to\infty}\lim_{N\to\infty}$.)
\end{remark}

\begin{proposition}
\label{prop:approx-crit-2}
    If $\Phi$ is a pseudo-maximizer for $\bbA$, then for any $\ul$-indexed sequence $(q_*,\ell)=\big((q_*,\ell)_{\ul\geq 1}\big)$ such that $q_*\in [q_1,1]$, $\ul\leq\ell\leq\ol$ and $\lim_{\ul\to\infty}|q_*-q_{\ell}^{\delta}|=0$, we have
    \[
    \lim_{\ul\to\infty}
    \plim_{N\to\infty}
    \lt\|\nabla H_N(\bn^{\ell})-\vA(q_*)
    \diamond \bn^{\ell}
    \rt\|_N
    =
    0.
    \]
\end{proposition}

\begin{proof}
    For notational convenience we assume $(q_*,\ell)=(1,\ol)$; the proof is identical in general.
    Recall the rearrangement \eqref{eq:amprearrange}:
    \begin{equation}
    \label{eq:copy-amp-rearrange}
    \nabla \widetilde{H}_N(\bn^{\ol})
    =
    \bz^{\ol+1}+\sum_{j=0}^{\ol} d_{\ol, j}\diamond \bn^{j-1}
    .
    \end{equation}
    So far we did not have to compute $d_{\ol, j}$. 
    We do this now, focusing on the IAMP phase. Recalling \eqref{eq:n-def}, the IAMP iteration used non-linearity
    \begin{align*}
    \f_{\ol}
    &=
    \bn^{\ol}
    =
    n^{\ul}+
    \sum_{j=\ul}^{\ol-1}
    (\bn^{j+1}-\bn^{j})
    \\
    &=
    \va\diamond\bz^{\ul}+
    \sum_{j=\ul}^{\ol-1}
    \bu^{\delta}_{j}
    \diamond
    (\bz^{j+1}-\bz^{j})
    \\
    &=
    (\va - \bu_{\ul}^{\delta})
    \diamond
    \bz^{\ul}
    + \bu_{\ol-1}^{\delta}\diamond \bz^{\ol}
    -
    \sum_{j=\ul+1}^{\ol-1}
    (\bu^{\delta}_{j}-\bu^{\delta}_{j-1})
    \diamond
    \bz^{j}
    .
    \end{align*}
    Using the formula \eqref{eq:dts-def} we find
    \begin{align*}
    d_{\ol,j,s}
    &\approx
    \begin{cases}
    \sum_{s'\in\sS}
    \partial_{s'}\xi^s(\Phi(1))
    ~
    u^{\delta}_{\ol,s'}
    ,\quad\quad\quad\quad\quad\quad j=\ol
    ;
    \\
    -
    \sum_{s'\in\sS}
    \partial_{s'}\xi^s(\Phi(q_{j-1}^{\delta}))
    ~
    (u^{\delta}_{j,s'}-u^{\delta}_{j-1,s'})
    ,\quad \ul<j<\ol
    ;
    \\
    \sum_{s'\in\sS}
    \partial_{s'}\xi^s(\Phi(q_1))
    ~
    (a_{s'}-u^{\delta}_{\ul,s'})
    ,\quad\quad\quad j=\ul
    .
    \end{cases}
    \end{align*}
    Note that since $\Phi\in C^2([q_1,1])$ we have the uniform-in-$q_j^{\delta}$ approximations (recall \eqref{eq:f-def}):
    \begin{equation}
    \label{eq:d-j-s}
    \begin{aligned}
    u^{\delta}_{\ell,s}
    &\approx
    f_s(q_j^{\delta}),
    \\
    \frac{u^{\delta}_{\ell,s}-u^{\delta}_{\ell-1,s}}{\delta}
    &\approx
    \frac{\de }{\de q}
    \sqrt{\fr{\Phi'_s(q)}{(\xi^s \circ \Phi)'(q)}}
    \,
    \Bigg|_{q=q_j^{\delta}},
    \\
    a_s\approx u_{\ul,s}^{\delta}&\approx f_s(q_{\ul}^{\delta})\approx f_s(q_1)
    .
    \end{aligned}
    \end{equation}
    Substituting into \eqref{eq:copy-amp-rearrange}, we obtain
    \begin{equation}
    \label{eq:gradient-expansion-iamp-crits}
    \begin{aligned}
        \nabla H_N(\bn^{\ol})
        &=
        \bh+\bz^{\ol+1}
        +
        \sum_{j=\ul}^{\ol} 
        d_{\ol, j}\diamond \bn^{j-1}
        \\
        &=
        \va^{-1}\diamond \bm^{\ul}
        +
        \sum_{j=\ul}^{\ol}
        u^{\delta}_j
        \diamond 
        (\bn^{j+1}-\bn^j)
        +
        \sum_{j=\ul}^{\ol} 
        d_{\ol, j}\diamond \bn^{j-1}
        \\
        &\approx
        \Big(
        \va^{-1}
        +
        \sum_{j=\ul}^{\ol}
        d_{\ol,j}
        \Big)
        \diamond \bn^{\ul}
        +
        \sum_{j=\ul}^{\ol-1}
        \vC_j\diamond (\bn^{j+1}-\bn^j)
        ;
        \\
        C_{j,s}
        &\equiv
        a_s^{-1}
        +
        \sum_{k=j+1}^{\ul}
        d_{\ol,k,s}
        \\
        &\stackrel{\eqref{eq:d-j-s}}{\approx}
        f_s(q_j^{\delta})^{-1}
        +
        \sum_{s'\in\sS}
        \lt(
        \partial_{s'}
        \xi^s(\Phi(1))
        f_{s'}(1)
        -
        \int_{q_j^{\delta}}^1
        \partial_{s'}\xi^s(\Phi(q))
        \, 
        f_{s'}'(q)~\de q
        \rt)
        \\
        &\equiv
        \wh C_s(q_j^{\delta})
        .
    \end{aligned}
    \end{equation}
    Since the increments $(\bn^{j+1}-\bn^j)$ are orthogonal in the state evolution sense, it easily follows that the approximation of $C_{j,s}$ by $\wh C_{s}(q_j^{\delta})$ commutes with summation, i.e.
    \[
    \nabla H_N(\bn^{\ol})
    \approx 
    \Big(
    \va^{-1}
    +
    \sum_{j=\ul}^{\ol}
    d_{\ol,j}
    \Big)
    \diamond \bn^{\ul}
    +
    \sum_{j=\ul}^{\ol-1}
    \wh C(q_j^{\delta})\diamond (\bn^{j+1}-\bn^j)
    \]

    Note that we manifestly have $\wh C(1)=\vA(1)$.
    We claim the function $\wh C$ is constant on $[q_1,1]$.
    This is equivalent to showing that for each $s$ the function
    \[
     F_s(q)
    =
    \frac{1}{f_s(q)}
    +
     \lt(
        \int_{q_1}^q
        \sum_{s'\in\sS}
         \partial_{s'}\xi^s(\Phi(t))
        f_{s'}'(t)
        ~\de t
        \rt)
    \]
     is constant. Differentiating, it suffices to show
     \begin{equation}
     \label{eq:wts-weird-identity}
     \sum_{s'\in\sS}
     \partial_{s'}\xi^s(\Phi(q))
     f_{s'}'(q)
     \stackrel{?}{=}
     f_s'(q)/f_s(q)^2
     .
     \end{equation}
    Write $f_{s'}'(q)=\Psi(q)\Phi_{s'}'(q)$, where $\Psi$ is independent of $s$ since $\Phi$ solves the tree-descending ODE \eqref{eq:tree-descending-ode}. Then using the chain rule, the left-hand side of \eqref{eq:wts-weird-identity} equals
     \[
     \Psi(q)
     \sum_{s'\in\sS}
     \partial_{s'}\xi^s(\Phi(q))\cdot \Phi_{s'}'(q)
     =\Psi(q)(\xi^s\circ\Phi)'(q).
     \]
     Meanwhile the right-hand side of \eqref{eq:wts-weird-identity} is
     \[
    f_s'(q)/f_s(q)^2
    =
    \Psi(q)\Phi_{s}'(q)
    \cdot\frac{(\xi^s\circ\Phi)'(q)}{\Phi_s'(q)}
    =
    \Psi(q)(\xi^s\circ\Phi)'(q).
     \]
     Therefore $\wh C(q)=\vA(1)$ is constant as claimed.
     Finally it is clear that the $\bn^{\ul}$ coefficient in \eqref{eq:gradient-expansion-iamp-crits} approximately equals $\wh C(q_1)$ and hence also $\vA(1)$. Then \eqref{eq:gradient-expansion-iamp-crits} implies
    \[
         \nabla H_N(\bn^{\ol})
         \approx 
         \vA(1)\diamond\lt(\bn^{\ul}+\sum_{j=\ul}^{\ol-1} (\bn^{j+1}-\bn^j)\rt)
         \\
         =\vA(1)\diamond \bn^{\ol}
    \]
     which completes the proof.
\end{proof}

From the point of view of \cite{huang2023algorithmic}, the fact that $\|\nabla_{\sph} H_N(\bn^{\ol})\|_N\approx 0$ is to be expected. At least for $(\Phi;q_1)$ maximizing $\bbA$, if this were not true than an extra step of gradient descent would essentially suffice to reach energy strictly better than $\ALG$, contradicting the optimality in \cite[Theorem 1]{huang2023algorithmic}. 
However the radial derivative computation is interesting in its own right and lets us study the spherical Hessian around an output $\bsig$.
We believe that Corollary~\ref{cor:top-eigen-zero} below can be strengthened to hold with $\blambda_1$ rather than $\blambda_{\eps N}$. This seems to require a more precise Gaussian conditioning argument around $\cA(H_N)$ which we chose not to pursue.

\begin{corollary}
\label{cor:top-eigen-zero}
    With $\blambda_k$ the $k$-th largest eigenvalue of a symmetric real matrix,
    \begin{equation}
    \label{eq:approx-max-eigenvalue}
    \lim_{\ul\to\infty,\eps\to 0}
    \plim_{N\to\infty}
    \blambda_{\eps N}
    \lt(
    \nabla^2_{\sph} H_N(\bn^{\ol})
    \rt)
    =0.
    \end{equation}
\end{corollary}

\begin{proof}
    Fixing $\vA=\vA(1)$, the bulk spectral measure of 
    \begin{equation}
    \label{eq:artifical-hessian}
    \bW(\bx)=\nabla^2 H_N(\bx) - \vA\diamond \bx
    \end{equation}
    for deterministic $\bx\in\cS_N$ concentrates with rate function $N^2$ around a limiting spectral measure independent of $\bx$. By union-bounding over an $\delta\sqrt{N}$-net as in \cite[Proof of Lemma 3]{subag2018following}, it thus suffices to show \eqref{eq:approx-max-eigenvalue} at a point $\bx\in\cS_N$ independent of $H_N$, with $\bW(\bx)$ in place of $\nabla^2_{\sph} H_N(\bsig)$.
    This is purely a statement of random matrix theory and is shown in \cite[Proposition 5.18]{huang2023strong}.
\end{proof}

Notably Corollary~\ref{cor:top-eigen-zero} explains the equality $\ALG=E_{\infty}$ for pure models, which we derived manually in \cite{huang2023algorithmic}. 
Indeed for a pure model with $\xi=\prod_{i=1}^r x_i^{a_i}$, the energy and radial derivative are deterministically proportional:
\[
    \nabla_{\rd} H_N(\bx)
    =
    -H_N(\bx)\va\diamond \bx,\quad\forall \bx\in\cB_N.
\]
It follows (using again the $N^2$ large deviation rate for the spectral bulk) that there is a unique energy level $E_{\infty}$ at which critical points can have spherical Hessian obeying the conclusion of Corollary~\ref{cor:top-eigen-zero}. 
This is the definition of $E_\infty$ given in \cite{auffinger2013random,mckenna2021complexity}.

\subsection{Branching IAMP and Exponential Concentration}
\label{subsec:branch}

Here we modify the second stage of our IAMP algorithm (which requires $\vDelta=\vone$) to use external Gaussian randomness in a small number of increment steps. This allows the construction of an ultrametric tree of outputs with large constant depth and $\exp(cN)$ breadth, with pairwise overlaps given by $\Phi$. 
More precisely, for any finite ultrametric space $X=(x_1,\dots,x_M)$, $M=\exp(cN)$, of diameter at most $1-q_1$, branching IAMP outputs $(\bsig_1,\dots,\bsig_M)$ with
\[
    \plim_{N\to\infty}
    \max_{1\leq i,j\leq M}\tnorm{
        \vR(\bsig_i,\bsig_j)
        - \Phi\big(1-d_X(x_i,x_j)\big)
    }_{\infty}=0.
\]
We use an approach suggested in \cite{alaoui2020algorithmic} by injecting external Gaussian noise $\bg^{(i)}$ into the IAMP phase of the algorithm at depth $q_i\in (q_1,1)$. 
Importantly, this gives an explicit construction of $\exp(cN)$ approximate critical points of $H_N$ (with exponentially good probability) whenever there is an IAMP phase. 
A similar construction was used by one of us in \cite[Section 4]{sellke2021optimizing}. There the Gaussian noise was constructed artifically by preliminary iterates of AMP rather than from exogenous noise (due to the lack of a state evolution result incorporating independent gaussian vectors). This only enabled the construction of a large constant number of outputs rather than exponentially many.

Our branching IAMP proceeds as follows. We first apply Stage $\I$ with $\vDelta=\vone$ as before. We fix $q_1<q_2<\dots<q_m=1$ and let 
\[
    \ell^{\delta}_{q_i}=\ul+\lt\lceil \fr{q_i-q}{\delta}\rt\rceil+1
    ,
    \quad
    i\in [m]
    .
\]
We define $\bn^{\ell}$ with the same recursive formula as before, unless $\ell=\ell^{\delta}_{q_i}$ for some $i\in [m]$. For these cases, we define $\bg^{(1)},\dots,\bg^{(m)}\sim \cN(0,\bone_N)$ to be independent standard Gaussian vectors. Then we set:
\begin{equation}
\label{eq:branchIAMP}
    \bn^{\ell+1} 
    =
    \begin{cases}
    \bn^{\ell}
    +
    \sqrt{\xi^s\big(\Phi(q^{\delta}_{\ell_{q_i}^{\delta}+1})\big)-\xi^s\big(\Phi(q^{\delta}_{\ell_{q_i}^{\delta}})\big)}
    \diamond \bg^{(i)},
    \quad \quad\quad
    \ell
    =
    \ell^{\delta}_{q_i}\quad
    \text{for some } 
    i\in [m]
    \\
    \bn^{\ell}+ 
    u_{\ell}^{\delta}\diamond
    \left(
    \bz^{\ell+1}-\bz^{\ell}
    \right),
    \quad \text{ else}.
    \end{cases}
\end{equation}

The definition \eqref{eq:branchIAMP} naturally enables couplings for pairs of iterations. We say the iterations $\big(\bn^{\ell,1},\bn^{\ell,2}\big)_{\ell\geq 1}$ are $q_j$-coupled if their associated Gaussian vectors
\begin{align*}
    \bg^{(1,1)},\dots,\bg^{(m,1)}
    &\sim 
    \cN(0,\bone_N),
    \\
    \bg^{(1,2)},\dots,\bg^{(m,2)}
    &\sim 
    \cN(0,\bone_N)
\end{align*}
are coupled so that $\bg^{(i,1)}=\bg^{(i,2)}$ almost surely for $i<j$, and the variables are otherwise independent.

\begin{proposition}
\label{prop:branching-is-same}
    Let the iterations $\bn^{\ell,1},\bn^{\ell,2}$ be $q_j$ coupled as above, and let $\Phi$ be a pseudo-maximizer of $\bbA$ (recall Definition~\ref{def:pseudo-maximizer}). Then
    \begin{align}
    \label{eq:branching-energy}
    \lim_{\ul\to \infty}
    \plim_{N\to\infty}
    \frac{H_N(\bn^{\ol,a}_{\delta})}{N}
    &=
    \bbA(\Phi),\quad a\in\{1,2\}
    ;
    \\
    \label{eq:branching-crit}
    \lim_{\ul\to \infty}
    \plim_{N\to\infty}
    \lt\|\nabla H_N(\bn^{\ell,a})-\vA(1)
    \diamond \bn^{\ell,a}
    \rt\|_N
    &=
    0,\quad a\in\{1,2\}
    ;
    \\
    \label{eq:branching-overlap}
    \lim_{\ul\to \infty}
    \plim_{N\to\infty}
    \frac{
    \la
    \bn^{\ol,1}_{\delta}
    ,
    \bn^{\ol,2}_{\delta}
    \ra
    }
    {N}
    &=
    \Phi(q_j^{\delta})
    .
    \end{align}
\end{proposition}

\begin{proof}
    The analysis uses the slightly generalized state evolution given in Theorem~\ref{thm:mixedAMP}, which states that \eqref{eq:state-evolution-basic} continues to hold even in the presence of external randomness $\bg^{(i)}$. Modulo this point, the calculations are essentially identical. Indeed \cite{sellke2021optimizing} uses exactly the same calculations to analyze a slightly different formulation of branching IAMP (therein, the vectors $\bg$ are defined via negatively time-indexed AMP iterates to sidestep the lack of a generalized state evolution result). We therefore give only an outline below.

    The SDE description in \eqref{lem:BMlimit} is unchanged if one uses the slightly added generality of Theorem \ref{thm:mixedAMP} to incorporate the external Gaussian noise. (This Gaussian noise is scaled in \eqref{eq:branchIAMP} to achieve exactly the same effect as a usual iteration step.)
    The energy analysis of $H_N(\bn^{\ol})$ only changes on the $m$ modified steps which has negligible effect since $\delta\to 0$ as $\ul\to\infty$; similarly for $\nabla H_N(\bn^{\ol})$. Thus \eqref{eq:branching-energy} follows by the same proof as before.
    The proof of \eqref{eq:branching-overlap} is identical to \cite[Section 8]{sellke2021optimizing}.
\end{proof}

In Proposition~\ref{prop:high-prob} below we observe that concentration of measure implies Proposition~\ref{prop:branching-is-same} holds with exponentially high probability. Thus we can couple together $\exp(cN)$ branching IAMPs to construct a full ultrametric tree of large constant depth $m$ and breadth $\exp(cN)$. To do this, we fix $m$, take $\ul$ sufficiently large and then $\eta>0$ sufficiently small. Then with $K=\exp(\eta N)$, we consider a complete depth $m$ rooted tree $\cT$, with root defined to have depth $1$, such that each vertex at depths $1,\dots,m-1$ has $K$ children. Thus the leaf-set $L(\cT)$ is naturally indexed by $[K]^m$. For $v,v'\in L(\cT)$ we let $v\wedge v'\in \{1,2,\dots,m\}$ denote the height of their least common ancestor.
For each non-leaf $x\in V(\cT)$, label the edge from $x$ to its parent by an i.i.d. Gaussian vector $\bg^{(x)}\sim\cN(0,I_N)$. 
Then for each leaf $v\in L(\cT)$, using the $m$ Gaussian vectors along the path from the root of $\cT$ to $v$ yields branching IAMP output $\bsig^{(v)}$ for any $H_N$.

\begin{proposition}
\label{prop:high-prob}
    Proposition~\ref{prop:branching-is-same}
    holds with exponentially good probability in the following sense. Fix $m$ and $q_1<q_2<\dots<q_m=1$.
    For any $\eps>0$, for large enough $\ul$ there exists $\eta=\eta(\eps,\ul)>0$ such that for $N$ large enough, the following hold simultaneously across all $v,v'\in L(\cT)$ with probability at least $1-\exp(-\eta N)$:
    \begin{equation}
    \label{eq:exp-high-prob}
    \begin{aligned}
    |\bbA(\Phi)
    -
    H_N(\bn^{\ol,v})/N
    |
    &\leq
    \eps;
    \\
    \|\nabla H_N(\bn^{\ol,v})-
    \vA(1) 
    \diamond
    \bn^{\ol,v}\|_N
    &\leq
    \eps
    ;
    \\
    \lt\|\vR(\bn^{\ol,v},\bn^{\ol,v'})
    -
    \Phi(q_{v\wedge v'})
    \rt\|_N
    &\leq
    \eps
    .
    \end{aligned}
    \end{equation}
\end{proposition}

\begin{proof}
    As explained in \cite[Section 8]{huang2021tight}, the map $H_N\mapsto \bn^{\ol}$ agrees with a $C(\ul)$-Lipschitz function of the coefficients $\bG^{(k)}$ of $H_N$ except with probability $1-\exp(-cN)$.
    The same proof applies for $H_N\mapsto \bn^{\ol,v}$ as well since the external noise variables are also Gaussian.
    Concentration of measure on Gaussian space now ensures that the statements above hold with exponentially high probability for each fixed $(v,v')$. Union bounding over all such pairs for small enough $\eta$ implies the result.
\end{proof}

In particular, the last conclusion in \eqref{eq:exp-high-prob} shows that all $\exp(\eta N)$ constructed points have pairwise distance at least $\delta\sqrt{N}$ for $0<\delta<1-q_{m-1}$.
Thus for any sub-solvable model, with high probability there are exponentially many $\sqrt{N}/C(\xi)$-separated approximate critical points.
This is a converse to the main result of \cite{huang2023strong}, where we show that strictly super-solvable models enjoy a \emph{strong topological trivialization} property which rules out such behavior.

\begin{remark}
An alternative to branching IAMP, which is very natural from the point of view of our companion work \cite{huang2023algorithmic}, is to slightly perturb $H_N$ to a $(1-\eta)$-correlated function $H_N^{(\eta)}$.
Concentration of measure implies that the overlap
\[
    \vR\big(\cA(H_N),\cA(H_N^{(\eta)})\big)
\]
concentrates exponentially around a limiting value $R_{\delta,\ul,\eta}\in\bbR^r$. We expect that taking $\eta\to 0$ with $\delta,\ul$ in a suitable way enables $R_{\delta,\ul,\eta}\approx\Phi(q)$ for any desired $q\in [q_1,1]$. This corresponds to the fact that $p(q)$ for $q\in [q_1,1]$ for any $(p,\Phi;q_0)$ maximizing $\bbA$. However this approach seems more cumbersome to analyze explicitly.
\end{remark}

\begin{remark}
    \revedit{The construction in this section shows the quenched existence of $\exp(\eta N)$ well-separated \emph{approximate} critical points for strictly sub-solvable models.
    In \cite[Theorem 5.15]{huang2023strong} we use this fact to prove the number of exact critical points is exponentially large \emph{in expectation}.
    However we are unable to prove the quenched (i.e. high-probability) existence of $\exp(\eta N)$ exact critical points in strictly sub-solvable models. Showing that this is the case, or more generally identifying the quenched exponential order of the number of critical points, is an interesting direction for future work.}
\end{remark}

\subsection*{Acknowledgements}

B.H. was supported by an NSF Graduate Research Fellowship, a Siebel scholarship, NSF awards CCF-1940205 and DMS-1940092, and NSF-Simons collaboration grant DMS-2031883.
M.S. was supported by an NSF
graduate research fellowship, a Stanford graduate fellowship, and NSF
award CCF-2006489 and was a member at the IAS while parts of this work were completed.





\bibliographystyle{alpha}
\bibliography{all-bib}

\appendix
\addtocontents{toc}{\protect\value{tocdepth}=1} 
\section{State Evolution: Proof of Proposition \ref{prop:state_evolution}}
\label{sec:ProofSE}

In this section we prove Proposition~\ref{prop:state_evolution}, following the appendix of \cite{ams20}. Throughout, we denote by $\bG^{(k)}\in (\bbR^N)^{\otimes k}$, $k\ge 2$ a sequence of standard Gaussian tensors. For $S_k$ the symmetric group on $k$ elements we also write 
\begin{equation}
\label{eq:Ak-def}
    \bA^{(k)}=\frac{1}{N^{(k-1)/2} }\Gamma^{(k)}
    \diamond 
    \sum_{\pi\in S^k}(\bG^{(k)})^{\pi}
\end{equation}
for the rescaled tensors with entries
\begin{equation}
\label{eq:A-defn}
     \bA^{(k)}_{i_1,\dots,i_k}
     =
     \frac{1}{N^{(k-1)/2} }
     \gamma_{s(i_1),\dots,s(i_k)} 
     \sum_{\pi\in S_k}\bG^{(k)}_{i_{\pi(1)},\dots,i_{\pi(k)}}
     \,.
\end{equation}
For a symmetric tensor $\bA^{(k)}\in(\bbR^N)^{\otimes k}$ and $\bT\in(\bbR^N)^{\otimes (k-1)}$, we denote by  $\bA^{(k)}\{\bT\}\in\bbR^N$ the vector with components
\begin{equation}
\label{eq:TensorTensor}
    \bA^{(k)}\{\bT\}_i = \frac{1}{(k-1)!}\sum_{1\le i_1,\cdots,i_{k-1}\le N}A^{(k)}_{i,i_1,\cdots,i_{k-1}}T_{i_1\dots i_{k-1}}.
\end{equation}
For $\bu\in\bbR^N$ we denote by $\bA^{(k)}\{\bu\}=\bA^{(k)}\{\bu^{\otimes (k-1)}\}$ the vector with entries
\begin{equation}
\label{eq:vector}
\bA^{(k)}\{\bu\}_i = \frac{1}{(k-1)!}\sum_{1\le i_1,\cdots,i_{k-1}\le N}A^{(k)}_{i,i_1,\cdots,i_{k-1}}u_{i_1}\cdots u_{i_{k-1}}.
\end{equation}
Note that for $\bA^{(k)}$ as in \eqref{eq:Ak-def}, one has 
\[
\bA^{(k)}\{\bu\}=\nabla H_{N,k}(\bu)
\]
where $H_{N,k}$ denotes the part of $H_N$ of total degree $k$.

For $\bu,\bv\in\bbR^N$ we recall from Subsection~\ref{subsec:notation} the notations 
\begin{align*}
    \< \bv\>_N&=N^{-1}\sum_{i\le N} v_i,
    \\
    \< \bu,\bv\>_N
    &= 
    N^{-1}\sum_{i\le N}u_iv_i
    =
    \langle\vec\lambda, \vR(\bu,\bv)\rangle,
    \\
    \|\bu\|_{N}&= \<\bu,\bu\>_N^{1/2}=\sqrt{\sum_s \lambda_s R_s(\bu,\bu)}.
\end{align*}

Given functions $f_{t,s}:\bbR^{t+1}\to\bbR$ of $t+1$
variables for each $s\in\sS$, and $\bv^0,\bv^1,\dots,\bv^t\in\bbR^{N}$, we define $f_t(\bv^0,\bv^1,\dots,\bv^t)\in\mathbb R^N$ component-wise via
\begin{align}
    f_t(\bv^0,\bv^1,\dots,\bv^t)_i=f_{t,s(i)}(v^0_i,\dots,v^t_i).
\end{align}
Finally, for a sequence of vectors $\bw^0,\bw^1,\dots$, we write $\bw^{\le t} = (\bw^0,\bw^1,\dots,\bw^t)$.

To deduce the state evolution result for mixed tensors, we analyze a slightly more general iteration where each homogenous
$k$-tensor is tracked separately, while restricting ourselves to the case where the mixture $\xi$ has finitely many components: $\gamma_{s_1,\dots,s_k} = 0$ for all $(s_1,\dots,s_k)\in \sS^k$ for all $k \ge D +1$ for some fixed $D \ge 2$. We then proceed by an approximation argument to extend the convergence to the general case $D = \infty$.

We begin by introducing the Gaussian process that captures the asymptotic behavior of AMP. 
Define $\xi^k$ to be the degree $k$ part of $\xi$, and 
\[
    \xi^{k,s}=\frac{1}{\lambda_s}\partial_{x_s}\xi^k(x_1,\dots,x_r)
\]
the degree $k-1$ part of $\xi^s$. 

An AMP iteration is specified by Lipschitz functions $f_{t,s}:\bbR^{2(t+1)}\to \bbR$ for each $(t,s)\in\naturals\times \sS$.\footnote{The unusual factor $2$ in the exponent comes from the external randomness vectors $\be^1,\dots,\be^t$.}  
For each iteration $t$, the state of the algorithm is given by vectors
$\bw^t\in \bbR^N$, and $\bz^{k,t}\in \bbR^N$, with $k\in\{2,\dots,D\}$. 
Moreover for each $t$, there is also an external randomness vector $\be^t\in\bbR^N$ with independent coordinates $e^t_i\sim \mu_{t,s(i)}$ from deterministic probability distributions $\big(\mu_{t,s}\big)_{t\geq 0,s\in \sS}$ with finite second moment.
We now start to define the AMP iteration steps (the definition finishes at \eqref{eq:recursive-covariance}). A single step is given by
\begin{align}
\label{eq:AMP-def1}
    \AMP_t
    \lt(\bw^0,\dots,\bw^t
    ;
    \be^0,\dots,\be^t\rt
    )_{k}
    &\equiv
    \bA^{(k)}
    \{
    \f_t
    \}
    -
    \sum_{t'\leq t}
    d_{t,t',k}
    \diamond 
    \f_{t'-1}
    \, ,
    \\
\label{eq:bff}
    \f_t &= f_t(\bw^0,\dots,\bw^t;\be^0,\dots,\be^t),
    \\
\label{eq:AMP-def2}
    d_{t,t',k,s} 
    &\equiv 
    \lt(
    \sum_{s'\in\sS}
    \partial_{x_{s'}}\xi^{k,s}
    \lt(\lt( 
    \E\lt[ 
    F_{t,s} F_{t'-1,s}
    \rt]
    \rt)_{s\in\sS}
    \rt) 
    \times
    \E\lt[ 
    \partial_{W^{t'}_{s'}} F_{t,s'}
    \rt]
    \rt)
    \, 
    ,
    \\
\label{eq:F-def}
    F_{t,s'}
    &\equiv
    f_{t,s'}\big(W^0_{s'},W^1_{s'},\dots, W^t_{s'};E^0_{s'},\dots,E^t_{s'}\big)
    .
\end{align}
A general multi-species tensor AMP algorithm then takes the form:
\begin{align}
\label{eq:TensorAMP}
    \bw^t=\sum_{2\leq k\leq D} \bz^{k,t}\, ,\quad \quad
    \bz^{k,t+1}
    =
    \AMP_t
    \lt(\bw^0,\dots,\bw^t;\be^0,\dots,\be^t\rt)_k\, .
\end{align}

For the right-hand side of \eqref{eq:F-def} to make sense, we must define for each $t\geq 0$ and $s\in \sS$ a distribution over sequences
$(W^0_s,\dots,W^t_s;E^0_s,\dots,E^t_s)$. The latter variables $E^{t'}_s\sim \mu_{t',s}$ are simply taken independent of each other and all other variables. The construction of the $W$ variables is recursive across $t$ as follows. 
For each $2\leq k\leq D$ and $s\in\sS$, we let $U^{k,0}_s\sim \nu_{k,s}$
and construct a centered Gaussian process
\[
    (U^{k,1}_s,U^{k,1}_s,\dots,U^{k,t}_s)
\]
which is independent of $U^{k,0}_s$. The variables $U^{k,t}_s$ and $U^{k',t'}_{s'}$ are independent unless $(k,s)=(k',s')$. It remains to specify the covariance of $(U^{k,t}_s)_{1\leq t\le T}$ which is given recursively by:
\begin{equation}
\label{eq:recursive-covariance}
\begin{aligned}
    \E\big[U^{k,t+1}_s U^{k,t'+1}_{s}\big]
    & = 
    \xi^{k,s}(\Sigma^{k,t,t'})\,;
    \\
    \Sigma^{k,t,t'}_{s}
    &=
    \bbE\lt[
    f_{t,s}(W^0_s,\dots,W^t_s;E^0_s,\dots,E^t_s)
    f_{t',s}(W^0_s,\dots,W^{t'}_s;E^0_s,\dots,E^{t'}_s)
    \rt],\quad \forall s\in \sS
    \\
    W^t_s & \equiv  \sum_{2\leq k\leq D} U^{k,t}_s\, 
    .
\end{aligned}
\end{equation}

The main result, an extension of Proposition \ref{prop:state_evolution}, follows. Below we use $\bbW_2$ to denote the Wasserstein-$2$ distance between probability measures on Euclidean space in any dimension. We say a function $\psi:\bbR^d\to\bbR$ is \textbf{pseudo-Lipschitz} if 
\[
    |\psi(\bw)-\psi(\by)|
    \leq 
    C(1+\|\bw\|+\|\by\|)
    \cdot 
    \|\bw-\by\|
    ,\quad\forall \bw,\by\in\bbR^d.
\]

\begin{theorem}[State Evolution for AMP]
\label{thm:mixedAMP}
Let $\{\bG^{(k)}\}_{k\ge 2}$ be independent standard Gaussian tensors with $\bG^{(k)}\in(\bbR^N)^{\otimes k}$, and define $\bA^{(k)}$ as in \eqref{eq:A-defn}. Fix a sequence of Lipschitz functions $f_{t,s}:\bbR^{k+1}\to\bbR$.
 Let $\bz^{2,0},\cdots \bz^{D,0}\in\bbR^N$ be deterministic vectors and $\bw^0 =\sum_{2\leq k\leq D} \bz^{k,0}$.
Assume that for each $s\in\sS$, the empirical distribution of the vectors 
\[
    (z_i^{2,0},\cdots z_i^{D,0}),\quad i\in \cI_s
\]
converges in $\bbW_2(\bbR^{D-1})$ distance to the law of the vector $(U^{k,0}_s)_{2\le k\le D}$. 

Let $\bw^{t}, \bz^{k,t}$, $t\ge 1$ be given by the \emph{tensor AMP} iteration. Then, for all $s\in\sS$ and $T\ge 1$ and for any pseudo-Lipschitz functions $\psi:\bbR^{D \times T}\to \bbR$ and $\wt\psi:\bbR^T\to\bbR$, we have
\begin{align}
\label{eq:z-U-convergence}
    \plim_{N\to\infty}\frac{1}{N_s}\sum_{i\in\cI_s}\psi \Big((z_i^{k,t})_{ k\le D,t\le T}\Big) 
    &= \E\lt[\psi\big((U^{k,t}_s)_{2\leq k\leq D,t\leq T}\big)\rt]\, 
    ;
    \\
\label{eq:w-W-convergence}
    \plim_{N\to\infty}\frac{1}{N_s}\sum_{i\in\cI_s}\wt\psi \Big((w_i^{t})_{t\le T}\Big) 
    &= \E\lt[\wt\psi\big((W^{t}_s)_{t\leq T}\big)\rt]\,
    .
\end{align}
\end{theorem}

Note that \eqref{eq:w-W-convergence} (which concerns the actual AMP iterates $\bw^t$) is a special case of \eqref{eq:z-U-convergence} (which is more convenient to prove). Indeed one can take $\psi\lt((z^{k,t})_{k\le D,t\le T}\rt)=\wt\psi\lt(\big(\sum_{k\le D}z^{k,t}\big)_{t\le T}\rt)$. 
In the special case that $c_k =0$ for all $k\ge D+1$, Proposition \ref{prop:state_evolution} follows immediately from Theorem~\ref{thm:mixedAMP} by baking the contribution of $\bh$ explicitly into $f_t$ (since we require $k\geq 2$ above). Proposition \ref{prop:state_evolution} for non-polynomial $\xi$ follows by a standard approximation argument outlined at the end of Subsection~\ref{subsec:further-def-app}.
For the remainder of this Appendix we thus focus on establishing \eqref{eq:z-U-convergence}.

\subsection{Further Definitions}
\label{subsec:further-def-app}

We define the notations
\begin{align*}
    \bW_t
    &=
    \big[
    \,\bw^0~|~\bw^1~|\;\cdots\;|~\bw^t
    \big]\, ,
    \\
    \bE_t
    &=
    \big[
    \,\be^0~|~\be^1~|\;\cdots\;|~\be^t
    \big]\, ,
    \\
    \bZ_{k,t}
    &=
    \big[
    \,\bz^{k,0}~|~\bz^{k,1}~|\,\cdots\;|~\bz^{k,t}\,
    \big].
\end{align*}
Given a $N\times (t+1)$ matrix such as $\bW_t$, and a tensor $\bA^{(k)}\in(\bbR^{N})^{\otimes k}$, we write 
$\bA^{(k)}\{\bW_t\}$ for the $N\times (t+1)$ matrix with columns $\bA^{(k)}\{\bw^0\}$, \dots, $\bA^{(k)}\{\bw^t\}$:
\begin{align*}
\bA^{(k)}\{\bW_t\}&=\Big[\bA^{(k)}\{\bw^0\}\Big|\bA^{(k)}\{\bw^1\}\Big|\;\cdots\;\Big|\bA^{(k)}\{\bw^t\}\Big]\, .
\end{align*}
We will write $\f_t=f_t(\W_t,\bE_t)=f_t(\bw^0,\dots,\bw^t,\be^0,\dots,\be^t)$ and also set 
\begin{align}
\label{eq:Ydef1}
    \by_{k,t+1}(\ZZ_{k,t})
    &=
    \A_k\lt\{\f_t\rt\} 
    = 
    \bz^{k,t+1}
    +
    \sum_{t_1\leq t} 
    d_{t,t_1,k}
    \diamond 
    \f_{t_1-1}\, ,
    \\
\label{eq:Ydef2}
    \bY_{k,t} 
    &= 
    [\by_{k,1}|\;\cdots\;|\by_{k,t}]
    \, ,
    \\
\nonumber
    \by_{t}(\ZZ_{k,t})&=\sum_{2\leq k\leq D}  \by_{k,t}(\ZZ_{k,t})\, . 
\end{align}
%
%
We also define an associated $(t+1)\times (t+1)$ Gram matrix $\bG_{\xi^{k,s}}=\bG_{\xi^{k,s},t}$ via 
\begin{equation}
\label{eq:bG-def}
    (\bG_{\xi^{k,s}})_{t_1,t_2}
    =
    \xi^{k,s}\big(\vR( 
    \f_{t_1},\f_{t_2})\big).
\end{equation}
The dependence of $\bG_{\xi^{k,s},t}$ on $t$ will often be suppressed (this dependence is relevant when inverting the matrix $\bG_{\xi^{k,s},t}$ but not for defining individual entries).
Finally, we let $\cF_t$ denote the $\sigma$-algebra generated by all iterates up to time $t$:
\begin{equation}
\label{eq:cFt-defn}
    \cF_t
    =
    \sigma\big(\{\bz_{k,t_1},\bw^{t_1},\be^{t_1},\f_{t_1}\}_{k\le D,t_1\le t}\big)\, .
\end{equation}

Throughout the proof of state evolution we make the following simplifying assumptions:

\begin{assumption}
\label{as:degree-D}
    $\xi$ is a degree $D$ polynomial with all coefficients $\gamma_{s_1,\dots,s_k}$ for $2\leq k\leq D$ strictly positive.
\end{assumption}

\begin{assumption}
\label{as:well-conditioned}
    Each matrix $\bG_{\xi^{k,s},t}$ is well-conditioned, i.e. 
    \[
    C^{-1}\le \sigma_{\min}(\bG_{\xi^{k,s},t})\le \sigma_{\max}(\bG_{\xi^{k,s},t})\le C
    \]
    for all $t\le T$. Here $\bG_{\xi^{k,s},t}$ is defined based on iterates that will appear in Theorems~\ref{thm:SELAMP} and \ref{lem:ampequalslamp}. The same holds for $\cL_{k,t}$ as defined in \eqref{eq:cL-k-t}.
\end{assumption}

It is a standard argument that to establish Proposition~\ref{prop:state_evolution}, it suffices to do so under the above assumptions. The reason is that one can always slightly perturb both $\xi$ and the non-linearities $f_{t,s}$ to ensure the assumptions hold. Then suitable continuity properties suffice to transfer all asymptotic guarantees. We refer the reader to \cite[Appendices A.8 and A.9]{ams20} for the arguments in the single-species case, still in the generality of mixed tensors. (In the more common setting $D=2$ of just a random matrix this step is also common for state evolution proofs, see e.g. \cite[Section 4.2.1]{javanmard2013state}.) The corresponding extension in our setting is completely analogous and omitted.

\subsection{Preliminary Lemmas}

The next lemma has several parts. All are elementary Gaussian calculations so their proofs are omitted.

\begin{lemma}
\label{lem:4}
For any deterministic $\bu,\bv\in \bbR^N$ and $\bA^{(k)}$ defined by \eqref{eq:A-defn} we have:
\begin{enumerate}
    \item Letting $g_0\sim\normal(0,1)$ be independent of $\bg\sim\normal(0,\id_N)$, we have
%
\begin{align}
    \bA^{(k)}\{\bu\}\ed
    \sum_{s\in\sS}
    \bg_s
    \sqrt{\xi^{k,s}(\vR(\bu,\bu))}
    +  
    \frac{g_0}{\sqrt N}
    \sum_{s\in\sS}
    \bu_s
    \sqrt{
        \sum_{s'\in\sS}
        \partial_{x_{s'}} \xi^{k,s}
        \Big(
        \vR(\bu,\bu)
        \Big)
    }
    \, .
\end{align}
\item Let $g_0,g_1,\dots,g_r\sim\normal(0,1)$ be independent. We have (jointly across $s\in\sS$)
%
\begin{align}
    \sqrt{\lambda_s N} R_s(\bv,\bA^{(k)}\{\bu\})
    \ed 
    \sqrt{\xi^{k,s}(\vR(\bu,\bu))\cdot \vR(\bv,\bv)}g_s
    +  
    \sqrt{
        \sum_{s'\in\sS}
        \partial_{x_{s'}} \xi^{k,s}
        \Big(
        \vR(\bu,\bu)
        \Big)
    }
    R_s(\bu,\bv) \, g_0\, .
\end{align}
\item For $s\in\sS$:
 \[
    R\lt(\bA^{(k)}\{\bu\},\bA^{(k)}\{\bv\}\rt)_s\simeq \xi^{k,s}\lt(\vR( \bu,\bv)\rt).
\]
%
\item For a deterministic symmetric tensor $\bT\in(\bbR^N)^{\otimes k-1}$, the vector
$\bA^{(k)}\{\bT\}$ is centered Gaussian. Its covariance is given by
\begin{align*}
  \E\big[\bA^{(k)}\{\bT\}_i\bA^{(k)}\{\bT\}_j\big] 
  &= 
  \langle \xi^{k,s(i)}\diamond \bT,\, \bT\rangle_N
  \cdot 1\{i=j\} 
  \\
  &
  \quad
  +\frac{k(k-1)}{N^{k-1}}\,
  \sum_{i_1,\dots,i_{k-2}=1}^N
  \gamma_{i,i_1,\dots,i_{k-2}}
  \gamma_{j,i_1,\dots,i_{k-2}}
  T_{i,i_1,\dots,i_{k-2}}
  T_{j,i_1,\dots,i_{k-2}}\, . 
\end{align*}
\item Let $\bP\in\bbR^{N\times N}$ be the orthogonal projection onto a (deterministic) subspace $S\subseteq\bbR^N$ with $d=\dim(S)=O(1)$. Then
\[
    \|\bP\bG^{(k)}\{\bu\} - \bG^{(k)}\{\bu\}\|_2 /\|\bG^{(k)}\{\bu\}\|_2\simeq 0.
\]
\end{enumerate}
\end{lemma}

We next develop a formula for the conditional expectation of a Gaussian tensor $\bA^{(k)}$ given
a collection of linear observations.
We set $\bD$ to be the $t\times t\times t$ tensor with entries $D_{ijk}=1$ if $i=j=k$ and $D_{ijk}=0$ otherwise.
\begin{lemma}
\label{lem:symregression}
Recalling \eqref{eq:cFt-defn}, let $\E\{\bA^{(k)}|\cF_t\}$.
Equivalently $\E\{\bA^{(k)}|\cF_t\}$ is the conditional expectation of $\bA^{(k)}$ given the linear-in-$\bA^{(k)}$ observations
\begin{align}
\bA^{(k)}\{\f_{t'}\}=\by_{k,t'+1} \, \;\;\; \mbox{ for $s\in\{0,\dots,t-1\}$.} \label{eq:LinearConstraint}
\end{align}
Then we have for $i_1,i_2,\dots,i_k\leq n$,  
\begin{align}
\label{eq:SymmRegression}
    \E[\bA^{(k)}|\cF_t]_{i_1,i_2,\dots,i_k}= 
    \sum_{j=1}^k \sum_{0\leq t_1,t_2\leq t-1} (\hZZ_{k,t})_{i_j,t_2}\cdot (\bG^{-1}_{\xi^{k,s},t-1})_{t_2,t_1}\cdot (\f_{t_1,i_1}\cdots \f_{t_1,i_{j-1}}\f_{t_1,i_{j+1}}\cdots \f_{t_1,i_{k}})\, .
\end{align}
Here, the matrix $\hZZ_{k,t}\in\bbR^{N\times t}$ is defined as the solution of a system of linear equations as follows.
Define the linear operator $\cT_{k,t}:\bbR^{N\times t}\to\bbR^{N\times t}$ by letting, for $i\leq N$, $0\leq t_3\leq t-1$:
\begin{align}
\label{eq:Tdef}
    [\cT_{k,t}(\bZ)]_{i,t_3}
    &= 
    \sum_{j=1}^N 
    \sum_{0\leq t_1,t_2\leq t-1} 
    (\f_{t_2})_{i} (\f_{t_2})_{j}
    \lt(
    (\bG_{\xi^{k,s(i)},t-1}^{-1})_{t_2,t_1} 
    \partial_{s(j)}\xi^{k,s(i)}\big(\vR(\f_{t_2},\f_{t_3})\big)
    \rt)
    \diamond
    (\bZ)_{j,t_1}\,
    .
\end{align}
Then $\hZZ_{k,t}$ is the unique solution of the following linear equation (with $\bY_{k,t}$ defined as per \eqref{eq:Ydef1})
\begin{align}
\hZZ_{k,t}+\cT_{k,t}(\hZZ_{k,t}) =\bY_{k,t}.\label{eq:ZZ-eq}
\end{align}

(Here, $\hZZ_{k,t} = [\hat{\bz}_{k,0},\cdots,\hat{\bz}_{k,t-1}]$ and $\bY_{k,t} = [\hat{\by}_{k,1},\cdots,\hat{\by}_{k,t}]$ have dimensions $N \times t$.)
\end{lemma}
%

The above formulas for $\E[\bA^{(k)}|\cF_t]$ and $\cT_{k,t}$ are rather complicated. In \cite[Appendix A]{ams20} the reader may find helpful tensor network diagrams for the single-species case. Unfortunately it is less clear how to draw a corresponding tensor network with multiple species.

\begin{proof}[Proof of Lemma \ref{lem:symregression}]
Let $\cV_{k,t}$ be the affine space of symmetric tensors satisfying the constraint \eqref{eq:LinearConstraint}.
The conditional expectation $\E[\bA^{(k)}|\cF_t]$ is the tensor with minimum weighted Frobenius norm $\|\cdot\|_{F,\xi^k}$ in the affine space $\cV_{k,t}$, given by
\begin{equation}
\label{eq:Gammak-inner-product}
    \|\bA\|_{F,\xi^k}^2
    =
    \langle 
    (\Gamma^{(k)})^{-1}\diamond \bA
    ,
    (\Gamma^{(k)})^{-1}\diamond\bA
    \rangle.
\end{equation}
Here $(\Gamma^{(k)})^{-1}$ is the entry-wise inverse of $\Gamma^{(k)}$, which exists by Assumption~\ref{as:degree-D}.

By Lagrange multipliers, there exist vectors $\bm^1,\dots,\bm^t\in\bbR^N$ such that $\E[\bA^{(k)}|\cF_t]=\hbA^{(k)}$ equals
\begin{align}
    \hbA^{(k)}_t \equiv\Gamma^{(k)}\diamond
    \sum_{t'=0}^{t-1}\sum_{j=1}^k \underbrace{\f_{t'}
    \otimes 
    \cdots 
    \otimes 
    \f_{t'}}_{\mbox{$j-1$ times}}\otimes\bm^{t'} 
    \otimes
    \underbrace{\f_{t'}\otimes \cdots \otimes \f_{t'}}_{\mbox{$k-j$ times}}\, .
\end{align}

Also by Lagrange multipliers, if a tensor $\hbA^{(k)}$ is of this form (for some choice of vectors $\bm^1,\dots,\bm^t$) and satisfies the constraints $\hbA^{(k)}\{\f_{t'}\}=\by_{k,t'+1}$
for $s< t$, then this tensor is unique and equals $\E[\bA^{(k)}|\cF_t]$.
Without loss of generality, we write
\begin{align}
\bm^{t_1}_i =  \sum_{t_2=0}^{t-1}(\bG^{-1}_{\xi^{k,s(i)},t-1})_{t_1,t_2}\hbz^{t_2}_i\, , \;\;\;\; 
\hZZ_{k,t}= [\hbz^1|\;\cdots\;|\hbz^t]\, .
\end{align}
By direct calculation we obtain that for each $i\in [N]$,
\begin{align}
\nonumber
    \big(\hbA^{(k)}_t\{\f_{t_1}\}\big)_i &= \sum_{t_2=0}^{t-1} (\bG_{\xi^{k,s(i)},t-1})_{t_1,t_2}(\bm^{t_2})_i
    + 
    \sum_{t_2=0}^{t-1} 
    \sum_{s'\in\sS}
    \lt(\partial_{s'}
    \xi^{k,s(i)}(\vR(\f_{t_1},\f_{t_2}))
    R_{s'}(\f_{t_1},\bm^{t_2})
    \rt)
    (\f_{t_2})_i
    \\
\label{eq:SumT}f
    &= 
    (\hbz_{t_1})_i
    +
    \sum_{t_2=0}^{t-1} 
    \sum_{s'\in\sS}
    \lt(\partial_{s'}
    \xi^{k,s(i)}(\vR(\f_{t_1},\f_{t_2}))
    R_{s'}(\f_{t_1},\bm^{t_2})
    \rt)
    (\f_{t_2})_i
    \, .
\end{align}

%
We next stack these vectors as columns of an $N\times t$ matrix. The first term yields $\hZZ_{k,t}$. Moreover the second term
coincides with $\cT_{k,t}(\hZZ_{k,t})$ by rearranging the order of sums in \eqref{eq:SumT}. Hence
\begin{align}
    \big[\hbA^{(k)}_t\{\f_{0}\},\cdots,\hbA^{(k)}_t\{\f_{t-1}\}\big] 
    & =  
    \hZZ_{k,t} + \cT_{k,t}(\hZZ_{k,t})\, .
\end{align}
This in turn implies that the equation determining $\hZZ_{k,t}$ takes the form \eqref{eq:ZZ-eq}.
\end{proof}

\subsection{Long AMP}
\label{sec:LAMP}

As an intermediate step towards proving Theorem \ref{thm:mixedAMP}, we introduce a new iteration that we call Long AMP (LAMP), following 
\cite{berthier2019state}. This iteration is less compact but simpler to analyze. For each $k\le D$, let  
$\cS_{k,t}\subseteq (\bbR^N)^{\otimes k}$ be the linear subspace of tensors $\bT$ that are symmetric and such that
$\bT\{\f_{t_1}\}=0$ for all $t_1<t$.   We denote by $\cP_t^{\perp}(\bA^{(k)})$ the projection of $\bA^{(k)}$ 
onto $\cS_{k,t}$, in the inner product space \eqref{eq:Gammak-inner-product} corresponding to $\Gamma^{(k)}$. 
We then define the LAMP mapping 
\begin{align}
\label{eq:LAMP1}
    \LAMP_t\lt(\bv^{\le t}\rt)_{k}
    &\equiv
    \cP_t^{\perp} (\bA^{(k)})
    \{\f_t\}
    +
    \sum_{0\leq t_1\leq t}  
    h_{t,t_1-1,k}
    \diamond 
    \qq^{k,t_1},
    \\
\label{eq:LAMP2}
    h_{t,t_1,k,s}
    &\equiv 
    \sum_{0\le t_2\le t-1}
    \big[\bG_{\xi^{k,s},t-1}^{-1}\big]_{t_1,t_2}
    \big[\bG_{\xi^{k,s},t}\big]_{t_2,t}
    ,~~~ h_{t,-1,k}=0. 
\end{align}
Here we use similar notations $\f_t = f_t(\VV_t;\bE_t)$ and $\bG_{\xi^{k,s},t}$ as before (recall \eqref{eq:bG-def}), and take the vectors $\be^t$ as before. However the quantities $\f_t,\bG_{\xi^{k,s},t}$ are now different: they are computed using the vectors $\bv^0,\dots,\bv^t$ using the recursion:
\begin{align}
    \bv^t=\sum_{2\leq k\leq D} \qq^{k,t}\, ,\;\;\;\;\;\;\;
    \qq^{k,t+1}=\LAMP_t\lt(\bv^{\le t}\rt)_{k}\, . 
\end{align}

Following \cite{berthier2019state,ams20}, we first establish state evolution for LAMP (under the non-degeneracy Assumption~\ref{as:degree-D}), and then deduce the result for the original AMP.
In analyzing LAMP we use notations analogous to the ones introduced for AMP. In particular:
\begin{align}
    \VV_{t}
    &=
    [\bv_{1}|\bv_{2}|\dots|\bv_{t}]
    \\
    \bQ_{k,t}
    &=
    [
    \qq_{k,1}^{\otimes k}|
    \qq_{k,2}^{\otimes k}|
    \dots|
    \qq_{k,t}^{\otimes k}
    ].
\end{align}

\subsection{State Evolution for Long AMP}

\begin{theorem}
\label{thm:SELAMP}
  Under the assumptions of Theorem \ref{thm:mixedAMP}, let $\qq^{2,0},\cdots \qq^{D,0}\in\bbR^N$
be deterministic vectors and $\bv^0 =\sum_{2\leq k\leq D} \bq^{k,0}$.
Assume that the uniform empirical distribution of the $N$ vectors $\{(q_i^{2,0},\cdots, q_i^{D,0})\}_{i\leq N}$ converges 
in $\bbW_2$ distance to the law of the vector $(U^{k,0})_{2\le k\le D}$.

Further we assume there is a constant $C<\infty$ such that for all $t\le T$:
\begin{itemize}
\item[$(i)$] The matrices $\bG_{\xi^{k,s},t}= \bG_{\xi^{k,s},t}(\VV)$ are uniformly well-conditioned as guaranteed by Assumption~\ref{as:well-conditioned}.
\item[$(ii)$] Let the linear operator $\cT_{k,t}:\bbR^{N\times t}\to\bbR^{N\times t}$ be defined as per \eqref{eq:Tdef}, with $\bG_{\xi^{k,s},t} = \bG_{\xi^{k,s},t}(\VV,\bE)$,
and  $\f_t=f_t(\VV,\bE)$, and define 
\begin{equation}
\label{eq:cL-k-t}
\cL_{k,t} = {\boldsymbol 1}+\cT_{k,t}.
\end{equation}
Then $C^{-1}\le \sigma_{\min}(\cL_{k,t})\le \sigma_{\max}(\cL_{k,t})\le C$.
\end{itemize}

Then the following statements hold for any $t\le T$ and sufficiently large $N$:
\begin{enumerate}[label=(\alph*)]
\item Correct conditional law: 
\begin{equation}\label{eq:conditional}
\qq^{k,t+1}|_{\mathcal F_t}\ed \E[\qq^{k,t+1}|\mathcal F_t] +  \cP_t^{\perp}(\tbA^{(k)}) \{\f_t\}\, .
\end{equation}
where $\tbA^{(k)}$ is a symmetric tensor distributed identically to $\bA^{(k)}$ and independent of everything else, and
$\cP_{t}^{\perp}$ is the projection onto the subspace $\cS_{k,t}$ defined in Section \ref{sec:LAMP}.
Further 
\begin{equation}
\label{eq:CondExp}
    \E[\qq^{k,t+1}|\mathcal F_t]
    = 
    \sum_{s\in\sS}
    \sum_{0\leq t_1\leq t}  
    h_{t,t_1-1,k,s}
    \qq^{k,t_1}_s\, .
\end{equation}
Moreover, the vectors $(\qq^{k,t+1})_{2\leq k\leq D}$ are conditionally independent given $\mathcal F_t$. 
\item Approximate isometry: we have
\begin{align}
\label{eq:c1}
    R_s(\qq^{k,t_1+1},\qq^{k,t_2+1})
    &\simeq 
    \xi^{k,s}\lt(\vR( \f_{t_1},\f_{t_2})\rt)\, ,
    \\
\label{eq:c3}
    R_s(\bv^{t_1+1},\bv^{t_2+1})
    &\simeq 
    \xi^{s}\big(\vR( \f_{t_1},\f_{t_2})\big). 
\end{align}
Moreover, both sides converge in probability to constants as $N\to\infty$, and for $k_1\neq k_2$ and any $(t_1,t_2)$ and $s\in\sS$,    
\begin{equation}
\label{eq:c2}
    R_s(\qq^{k_1,t_1},\qq^{k_2,t_2})\simeq 0.
\end{equation}
\item  State evolution: for each $s\in\sS$ and any pseudo-Lipschitz function 
$\psi:\bbR^{D \times 2(t+1)}\to \bbR$, we have
\begin{align}
\label{eq:ConvergenceLAMP}
    \plim_{N\to\infty}
    \frac{1}{N_s}
    \sum_{i\in\cI_s}
    \psi\big((q_i^{k,t'})_{ k\le D,t'\le t}; (e^t_i)_{t'\leq t}\big) 
    = 
    \E\big\{\psi\big((U^{k,t'}_s)_{2\leq k\leq D,t'\leq t}; 
    (E^{t'}_s)_{t'\leq t}
    \big)\big\}\, .
\end{align}
where $(U^{k,t}_s)_{k\le D,1\le t\le T}$ is the centered Gaussian process defined in the statement of  Theorem \ref{thm:mixedAMP}.
\end{enumerate}
\end{theorem}
In the next subsection, we will prove these statements by induction on $t$. The crucial point we exploit is the representation $(a)$. We emphasize that the iteration number $t$ is bounded as $N\to\infty$; therefore all numerical quantities not depending on $N$ (but possibly on $t$) will be treated as constants.

\subsection{Proof of Theorem~\ref{thm:SELAMP}}

The proof will be by induction over $t$. The base case is clear, (e.g. see Proposition~\ref{prop:W2-converge-basic}) and we focus on the inductive step. We assume the statements above for $t-1$ and prove them for $t$. 

\subsubsection{Proof of $(a)$}

Note that $\cP_t^{\perp}(\bA^{(k)})$ is by construction independent of $\cF_t$, and therefore we can replace $\bA^{(k)}$
by a fresh independent matrix in \eqref{eq:LAMP1}, whence \eqref{eq:conditional} follows. The equality \eqref{eq:CondExp} holds by definition of the iteration.

\subsubsection{Proof of $(b)$: Approximate isometry}

We will repeatedly apply Lemma~\ref{lem:4}. We start with  \eqref{eq:c1}. As we are inducting on $t$, we may limit ourselves to considering overlaps
$\vR( \qq^{k,t+1},\qq^{k,t_1+1})$, for $t_1\le t$. 

Define the tensor $\Gamma^{(k),\nabla}\in (\bbR^{\sS}_{\geq 0})^{\otimes (k-1)}$ by
\begin{equation}
\label{eq:Gamma-nabla-def}
    \Gamma^{(k),\nabla}_{s_1,\dots,s_{k-1}}
    =
    \sqrt{
    k
    \sum_{s\in\sS}
    \lambda_s
    \big(\Gamma^{(k)}_{s,s_1,\dots,s_{k-1}}\big)^2
    }.
\end{equation}
We choose 
\[
    (\f_{t}^{\otimes k-1})_{\parallel}
    \in
    {\rm span}\lt(\f_{t_1}^{\otimes k-1}\rt)_{t_1< t}
\]
such that
\[
    \Gamma^{(k),\nabla}\diamond \big(\f_{t}^{\otimes k-1}\big)_{\parallel}
\]
is the orthogonal projection of $\Gamma^{(k),\nabla}\diamond \f_{t}^{\otimes k-1}$ onto 
\[
    {\rm span}\lt(\Gamma^{(k),\nabla}\diamond\f_{t_1}^{\otimes k-1}\rt)_{t_1< t}
\]
and also set
\[
    (\f_{t}^{\otimes k-1})_{\perp}=\f_{t}^{\otimes k-1}-(\f_{t}^{\otimes k-1})_{\parallel}.
\]

We will use (and soon after, prove) the following lemma.
\begin{lemma}
\label{lem:LemmaPerp}
For all $t_1\leq t_1$, we have
\begin{equation}
\label{eq:LemmaPerp2}
    \cP_t^{\perp} (\tbA^{(k)})\{(\f_t^{\otimes k-1})_{\perp}\} \simeq
    \tbA^{(k)}\{(\f_t^{\otimes k-1})_{\perp}\}\, .
\end{equation}
\end{lemma}
For $t_1\le t-1$, using Lemma~\ref{lem:4}, point 2 implies
\begin{align*}
    \vR(\qq^{k,t+1},\qq^{k,t_1+1}) 
    \simeq 
    \vR\big(\E[\qq^{k,t+1}|\cF_t],\, \qq^{k,t_1+1}\big)
\end{align*}
We next use the formula in $(a)$ for $\E[\qq^{k,t+1}|\cF_t]$ together with the expression in \eqref{eq:LAMP1}. For each $s\in\sS$:
\begin{align}
\nonumber
    R\big(\mathbb E[\qq^{k,t+1}|\mathcal F_t],\qq^{k,t_1+1}\big)_s
    &\simeq
    R\Bigg(
    \sum_{0\leq t_2,t_3\leq t-1}
    \qq^{k,t_3+1}(\bG_{\xi^{k,s},t-1}^{-1})_{t_3,t_2}
    \,
    \xi^{k,s}\big(\f_{t_2},\f_{t}\big)
    ,\, \qq^{k,t_1+1}
    \Bigg)_s
    \\
\nonumber
    &= 
     \sum_{0\leq t_2,t_3\leq t-1}
     R\big(\qq^{k,t_3+1},\qq^{k,t_1+1}\big)_s
     \,
     (\bG_{\xi^{k,s},t-1}^{-1})_{t_3,t_2} 
     \,
     \xi^{k,s}\big(\f_{t_2},\f_{t}\big)
    \\
\label{eq:3rd-ineq}
    &\simeq
    \sum_{0\leq t_2,t_3\leq t-1} 
    (\bG_{\xi^{k,s},t-1})_{t_3,t_1}
    \,
    (\bG_{\xi^{k,s},t-1}^{-1})_{t_3,t_2} 
    \,
    \xi^{k,s}\big(\f_{t_2},\f_{t}\big)
    \\
\nonumber
    &=
    \big(
    \bG_{\xi^{k,s},t-1}
    \times 
    \bG_{\xi^{k,s},t-1}^{-1}
    \times
    \bG_{\xi^{k,s},t-1}
    \big)_{t_1,t}
    \\
\nonumber
    &=
    (\bG_{\xi^{k,s},t-1})_{t_1,t}.
\end{align}
%
Here \eqref{eq:3rd-ineq} comes from the induction hypothesis \eqref{eq:c1} (and the symmetry of the matrix $\bG_{\xi^{k,s},t-1}$ is used to obtain the next line). We next prove that \eqref{eq:c1} holds for $t_1=t$. We have by definition of the projections that
\begin{align*}
\cP_t^{\perp} (\tbA^{(k)})\{\f_t\}=\cP_t^{\perp} (\tbA^{(k)})\{(\f_t^{\otimes k-1})_{\perp}\} \, ,
\end{align*}
where the right-hand side is defined according to \eqref{eq:TensorTensor}.
Using \eqref{eq:LemmaPerp2} from Lemma~\ref{lem:LemmaPerp} as well as point~4 of Lemma~\ref{lem:4}, we have
\begin{align}
\label{eq:PAnorm}
    R\lt(
    \cP_t^{\perp} (\tbA^{(k)})\{\f_t\}
    \,,
    \cP_t^{\perp} (\tbA^{(k)})\{\f_t\}
    \rt)_s
    \simeq 
    \xi^{(k,s)}\lt(
    R\big(
    (\f_{t}^{\otimes k-1})_{\perp}
    \,,
    (\f_{t}^{\otimes k-1})_{\perp}
    \big)
    \rt).
\end{align}
Next, using \eqref{eq:LemmaPerp2} and Lemma~\ref{lem:4} (point 2), we obtain that for all $s\in\sS$
\begin{equation}
\label{eq:ApproxOrth}
    R\big(\cP_t^{\perp} (\tbA^{(k)})\{\f_t\},\E[\qq^{k,t+1}|\cF_t]\big)_s \simeq 0\, .
\end{equation}
Moreover we recall that by the expression for $\E[\qq^{k,t+1}|\cF_t]$ from part $(a)$,
\begin{align}
    R\lt(
    \E[\qq^{k,t+1}|\mathcal F_t]
    \,,
    \E[\qq^{k,t+1}|\mathcal F_t]
    \rt)
    \simeq
    \xi^{k,s}
    \lt(
    \vR(
    (\f_t^{\otimes k-1})_{\parallel}
    \,,
    (\f_t^{\otimes k-1})_{\parallel}
    )
    \rt).
\end{align}
The formula for linear regression implies
\begin{align}
\label{eq:lin-reg-1}
    (\f_{t}^{\otimes k-1})_{\parallel}
    & =
    \sum_{0\leq t_1\leq t-1} \alpha_{t_1,t} \Gamma^{(k,s)}\diamond \f_{t_1}^{\otimes k-1},
    \\
    \alpha_{t_1,t}
    &=
    \sum_{0\leq t_2\le t-1} 
    (\bG_{\xi^{k,s},t-1}^{-1})_{t_1,t_2}
    \< 
    \Gamma^{(k,s)}\diamond \f_{t_2}^{\otimes k-1},
    \Gamma^{(k,s)}\diamond \f_{t}^{\otimes k-1}
    \>_N
    \\
    &=
    \sum_{0\leq t_2\le t-1} 
    (\bG_{\xi^{k,s},t-1}^{-1})_{t_1,t_2}
   (\bG_{\xi^{k,s},t})_{t_2,t}
    \, .
 \end{align}

By part $(b)$ of the inductive step, for $1\leq t_1,t_2\leq t-1$ we have
\[
    \xi^{k,s}\big(\vR( \f_{t_2},\f_{t_1})\big)
    \simeq 
    R_s( \qq_{k,t_2+1},\qq_{k,t_1+1})
    \,.
\]
In particular the formulas \eqref{eq:LAMP2} and \eqref{eq:lin-reg-1} have asymptotically the same coefficients, and the overlap structure between the summands is identical. It follows that
\begin{align}
\label{eq:CEnorm}
    R\lt(
    \E[\qq^{k,t+1}|\mathcal F_t],
    \,,
    \E[\qq^{k,t+1}|\mathcal F_t]
    \rt)
    \simeq
    R\lt(
    (\f_{t}^{\otimes k-1})_{\parallel}
    \,,
    (\f_{t}^{\otimes k-1})_{\parallel}
    \rt).
\end{align}
Using together Eqs.~\eqref{eq:PAnorm}, \eqref{eq:ApproxOrth}, and \eqref{eq:CEnorm}, we get
\begin{align*}
    R\lt(\qq^{k,t+1},\qq^{k,t+1}\rt)
    &\simeq 
    R\lt(\E[\qq^{k,t+1}|\mathcal F_t]
    \,,
    \E[\qq^{k,t+1}|\mathcal F_t]
    \rt)
    + 
    R\lt(
    (\f_t^{\otimes k-1})_{\perp}
    \,,
    (\f_t^{\otimes k-1})_{\perp}
    \rt)
    \\
    &\simeq
    R\lt(
    (\f_t^{\otimes k-1})_{\perp}
    \,,
    (\f_t^{\otimes k-1})_{\perp}
    \rt)
    \\
    &=
    \xi^{k,s}\lt(
    \vR( \f_t,\f_t)
    \rt)
    \,.
 \end{align*}
This establishes \eqref{eq:c1}.

Next consider \eqref{eq:c2}, i.e., approximate orthogonality of $\qq^{k,r}$ and $\qq^{p',r}$ for $k\neq p'.$ This follows easily from the representation
in point $(a)$ which, together with Lemma~\ref{lem:4}, inductively implies that the iterates $\qq^{s,k}$ for different $k$ are approximately orthogonal.
Finally, \eqref{eq:c3}  follows directly from \eqref{eq:c1} and \eqref{eq:c2}.
We now prove Lemma~\ref{lem:LemmaPerp}.

\begin{proof}[Proof of Lemma~\ref{lem:LemmaPerp}]
For convenience we write $\tbA=\tbA^{(k)}$. By Lagrange multipliers, there exist vectors
    $(\btheta_{t_1})_{t_1 \le t-1}$ in $\bbR^N$ such that $\cP_t^{\perp} (\tbA) = \tbA - \bQ$, where 
\begin{align*}
  \bQ = 
  \big(\Gamma^{(k)}\big)^{\odot 2}
  \diamond
  \frac{(k-1)!}{N^{k-1}}\sum_{t_1=0}^{t-1} \sum_{j=1}^k  \underbrace{\f_{t_1}\otimes \cdots \otimes \f_{t_1}}_{\mbox{$j-1$ times}}\otimes \btheta_{t_1} \otimes \underbrace{\f_{t_1}\otimes \cdots \otimes \f_{t_1}}_{\mbox{$k-j$ times}}.
\end{align*}
The vectors $(\btheta_{t_1})_{t_1 \le t-1}$ are determined by the equations
$\bQ \{\f_{t_1}\}=\tbA \{\f_{t_1}\}$
for all $t_1\le t-1$.
This expands (for each $t_1\leq t-1$) to
\begin{align*}
    \sum_{t_2\leq t-1} 
    (\bG_{\xi^{k,s},t-1})_{t_1,t_2}
    \diamond 
    \btheta_{t_2}
    +
    \sum_{t_2\leq t-1}
    \sum_{s'\in\sS}
    \lt(\partial_{s'}
    \xi^{k,s(i)}(\vR(\f_{t_1},\f_{t_2}))
    R_{s'}(\f_{t_1},\btheta_{t_2})
    \rt)
    \f_{t_2}
    = 
    \tbA\{\f_{t_1}\}\, .
\end{align*}
Recall that we assume each $\bG_{\xi^{k,s},t-1}$ is well-conditioned with high probability. Thus we can multiply the system of $t$ equations above by $\bG_{\xi^{k,s},t-1}^{-1}$ in the coordinates $\cI_s$ for each $s\in\sS$. For each $t_3\leq t-1$, we obtain:
\begin{equation}
\begin{aligned}
\label{eq:LambdaEq}
  &\btheta_{t_3}
  +
  \sum_{t_1,t_2<t}
  \lt(
  (\bG_{\xi^{k,s},t-1}^{-1})_{t_1,t_3} 
  \sum_{s'\in\sS}
    \lt(\partial_{s'}
    \xi^{k,s(i)}(\vR(\f_{t_1},\f_{t_2}))
    R_{s'}(\f_{t_1},\btheta_{t_2})
    \rt)
  \rt)
  \f_{t_2}
  \\
  &=
  \sum_{t_1<t}
  (\bG_{\xi^{k,s},t-1}^{-1})_{t_1,t_3}
  \diamond
  \tbA\{\f_{t_1}\}\, . 
\end{aligned}
\end{equation}
Switching $t_3$ to $t_1$, we find
\begin{align}
\nonumber
  \btheta_{t_1} &=   \btheta^0_{t_1}+ \btheta^{\parallel}_{t_1}\, ,
  \\
\label{eq:theta0}
  \btheta^0_{t_1}
  &\equiv
  \sum_{t_2<t}
  (\bG_{\xi^{k,s},t-1}^{-1})_{t_1,t_2}\diamond \tbA\{\f_{t_2}\} \, ,
  \\
\nonumber
  \btheta^{\parallel}_{t_1}
  &\in 
  \spn\lt((\f_{t_2,s})_{t_2<t,s\in\sS}\rt).
\end{align}
We claim that $\|\btheta^{\parallel}_{t_1}\|_N\simeq 0$, i.e.,
$\btheta_{t_1}\simeq   \btheta^0_{t_1}$. Indeed, let $\bTheta\in\bbR^{N\times t}$ be the matrix with columns
$(\btheta_{t_2})_{t_2<t}$, and $\bTheta^0$ the matrix with columns $(\btheta^0_{t_2})_{t_2<t}$.
Then \eqref{eq:LambdaEq} can be written as
\begin{align*}
  \cL_{k,t}^{\sT}(\bTheta) =\bTheta^0\, .
\end{align*}
Here we recall $\cL_{k,t}=\bfone+\cT_{k,t}$ and $\cT_{k,t}\in\bbR^{Nt\times Nt}$ is defined in \eqref{eq:Tdef}.
Substituting the decomposition 
$\bTheta = \bTheta^0+\bTheta^{\parallel}$ in the above, we obtain
\begin{align*}
  \cL_{k,t}^{\sT}(\bTheta^{\parallel}) =-\cT^{\sT}_{k,t}(\bTheta^0)\, .
\end{align*}
Recall that $\cL_{k,t}$ is well-conditioned by Assumption~\ref{as:well-conditioned}. Therefore it remains to prove 
\begin{equation}
\label{eq:key-lamp-computation}
    \cT^{\sT}_{k,t}(\bTheta^0)\stackrel{?}{\simeq} 0.
\end{equation}
Let $\bc_0,\cdots,\bc_{t-1} \in \bbR^N$ be the columns of $\cT^{\sT}_{k,t}(\bTheta^0)$. We first note that for all $t_1\leq t-1$ and $s\in\sS$, 
\[
    \bc_{t_1,s} \in \spn\big((\f_{t_2,s})_{t_2<t}\big).
\]
Moreover the Gram matrix 
\[
    \bG_{1,t-1,s}=\lt(R_s(\f_{t_1},\f_{t_2})\rt)_{t_1,t_2<t}
\]
is well-conditioned for each $s\in\sS$. Therefore it is sufficient to check that $R_s(\f_{t_1},\bc_{t_4})\simeq 0$ for each $t_1,t_4<t$ and $s\in\sS$. Plugging in the definition \eqref{eq:Tdef}, it remains to check
that for $0\leq t_1,t_4\leq t-1$,
\begin{align*}
    \sum_{t_2,t_3<t}
    \sum_{s''\in\sS}
    \lambda_{s''}
    \vR_s\lt(
    \f_{t_1}
    ,
    \,
    (\bG_{\xi^{k,s''},t-1}^{-1})_{t_2,t_3} 
    \sum_{s'}
    \partial_{s'}
    \xi^{k,s''}(\vR(\f_{t_4},\f_{t_3}))
    R_{s'}(\f_{t_4},\btheta^0_{t_2})
    \rt)
    R_s(\f_{t_1},\f_{t_3})
    \stackrel{?}{\simeq} 
    0\, 
    .
\end{align*}
Finally, this last claim follows by substituting the definition \eqref{eq:theta0} of $\btheta^0_{t_2}$, and using the fact that  
\[
    R_{s'}(\f_{t_4},\tbA\{\f_{t_2}\})\simeq 0,\quad
    \forall~ t_4,t_2\le t,\,s'\in\sS
\]
which follows from Lemma \ref{lem:4}. Thus \eqref{eq:key-lamp-computation} is established.

We are now ready to prove Lemma~\ref{lem:LemmaPerp}. 
First note that 
\begin{equation}
\label{eq:god}
     \tbA\{(\f_t^{\otimes k-1})_{\perp}\} 
    -
    \cP_t^{\perp} 
    (\tbA)
    \{(\f_t^{\otimes k-1})_{\perp}\} 
    = 
    \bQ\{(\f_t^{\otimes k-1})_{\perp}\}
\end{equation}
decomposes into two types of terms based on the definition of $\bQ$ above. Recalling \eqref{eq:Gamma-nabla-def}, the first involves
\[
    \lt\langle 
    \Gamma^{(k),\nabla}
    \diamond 
    (\f_{t_1}^{\otimes k-1})_{\perp}
    ,
    \Gamma^{(k),\nabla}
    \diamond 
    (\f_t^{\otimes k-1})_{\perp}
    \rt\rangle_N
    \btheta_{t_1}
\]
for $t_1\leq t-1$, which vanishes by the definition of $(\f_t^{\otimes k-1})_{\perp}$. The other terms take the form
\[
    \lt\langle 
    \Gamma^{(k),\nabla}
    \diamond 
    (\btheta_{t_1}\otimes \f_{t_1}^{\otimes k-2})
    ,\,
    \Gamma^{(k),\nabla}
    \diamond 
    (\f_t^{\otimes k-1})_{\perp}
    \rt\rangle_N
    ~
    \f_{t_1}.
\]
In particular, this means that to prove \eqref{eq:god} vanishes, suffices to show
\[
    R\lt(
    \bQ\{(\f_t^{\otimes k-1})_{\perp}\}
    ,\,
    \f_{t_2}
    \rt)
    =0
\]
for all $t_2\leq t$.

Note that by construction,
\[
   (\f_t^{\otimes k-1})_{\perp}
   =
   \sum_{t_1\leq  t} b_{t_1} \f_{t_1}^{\otimes k-1}
    \,.
\]
By the well-conditioning assumption, the $b_{t_1}$ are bounded. Therefore it suffices to show that 
\[
    R\lt(
    \bQ\{\f_{t_2}^{\otimes k-1}\}
    ,\,
    \f_{t_1}
    \rt)
    \stackrel{?}{\simeq} 
    0,
    \quad 
    \forall\, t_1\leq t-1,\,t_2\leq t.
\]
Finally note that each term in the left-hand side includes an overlap $R_s(\btheta_{t_1},\f_{t_2})$. However these all vanish:
\[
    R_s(\btheta_{t_1},\f_{t_2})
    \simeq
    0.
\]
This is because we can substitute $\btheta_{t_1}$ with $\btheta_{t_1}^0$ as defined in \eqref{eq:theta0} and use the fact that $\vR(\tbA\{\f_{t_3}\},\f_t)\simeq 0$ which follows from Lemma \ref{lem:4}. This completes the proof.
\end{proof}

\subsubsection{Proof of $(c)$}

The base case of initialization is handled by the following basic fact.
\begin{proposition}
\label{prop:W2-converge-basic}
    Let $\mu\in\cP(\bbR^k)$ be a probability distribution with finite second moment. Then if $E_1,\dots,E_N\stackrel{i.i.d.}{\sim}\mu$ and $\hat\mu_N=\frac{1}{N}\sum_{i=1}^N \delta_{E_i}$, one has
    \[
    \plim_{N\to\infty}
    \bbW_2(\hat\mu_N,\mu)=0.
    \]
\end{proposition}

\begin{proof}
    It suffices to show that $\hat\mu_N\to\mu$ weakly in probability and show convergence in probability of the $L^2$ norm. The first is clear and the second holds by the law of large numbers.
\end{proof}

Continuing to the inductive step, recall that the process $(U^{k,t}_s)_{t\ge 1}$ is Gaussian by construction, and independent of $U^{k,0}_s$. Define 
\begin{align*}
    C_{t_1,t_2,s} &= \E\big[U^{k,t_1}_s U^{k,t_2}_s\big]\,;
    \\
    \bC_{\le t,s} &= (C_{t_1,t_2,s})_{t_1,t_2\le t}\,.
\end{align*}
We then have
\begin{equation} 
\label{eq:wt-alpha}
\begin{aligned}
    \E[U^{k,t+1}_s| U^{k,0}_s,\dots,U^{k,t}_s]
    &=
    \sum_{t_1=1}^t 
    \wt\alpha_{t_1,s} U^{k,t_1}
    \, ;
    \\
    \wt\alpha_{t_1,s}
    & \equiv
    \sum_{t_2=1}^t
    (\bC^{-1}_{\le t,s})_{t_1,t_2}C_{t_2,t+1,s}
    \, .
\end{aligned}
\end{equation} 
Here in writing $(\bC^{-1}_{\le t,s})_{t_1,t_2}$, we view $\bC_{\le t,s}$ as a $(t+1)\times (t+1)$ matrix for each $s\in\sS$.

On the other hand, from point $(a)$, we know that
\begin{equation} 
\label{eq:alpha}
\begin{aligned}
    \E[\qq^{k,t+1}_{s}|\mathcal F_t]
    & = 
    \sum_{1\leq t_1\leq t}  
    \alpha_{t_1,s} 
    \qq^{t_1,k}_s\, ;
    \\
    \alpha_{t_1,s}
    & \equiv
    \sum_{t_2=1}^t
    (\bG^{-1}_{\xi^{k,s},t-1})_{t_1-1,t_2-1} 
    (\bG_{\xi^{k,s},t})_{t_2-1,t}
    \, .
\end{aligned}
\end{equation} 
Moreover the induction hypothesis of \eqref{eq:ConvergenceLAMP} implies that for $t_1,t_2 \le t$,
\begin{equation}
\label{eq:augment-e-for-W2}
    (\bG_{\xi^{k,s},t})_{t_1,t_2} 
    \simeq 
    \bbE\lt[
    \xi^{k,s}
    \lt(
    f_{t_1}(W^0_s,\dots,W^{t_1}_s;E^0_s,\dots,E^{t_1}_s), f_{t}(W^0_s,\dots,W^{t_2}_s;E^0_s,\dots,E^{t_2}_s)\}
    \rt)
    \rt]
    \, .
\end{equation}
(Recall that by definition $W^t_s \equiv \sum_{k\le D} U^{k,t}_s$, while $\f_t=f_t(\VV_t;\bE_t)$ here.)

Therefore, from the definition of the process $(U^{k,t}_s)_{t\ge 0}$,
\[
    (\bG_{\xi^{k,s},t})_{t_1,t_2} \simeq C_{t_1+1,t_2+1,s}
    ,
    \quad\quad \forall t_1,t_2\le t.
\]
Recalling that $\bG_{\xi^{k,s},t}$ is well-conditioned, we find (recall \eqref{eq:wt-alpha},\eqref{eq:alpha}):
\[
    \alpha_{t_1,s}\simeq \wt\alpha_{t_1,s}.
\]
Therefore we also have 
\begin{align*}
    \E[\qq^{k,t+1}|\mathcal F_t] - \sum_{t_1=1}^t \wt\alpha_{t_1}\diamond \qq^{k,t_1}
    &=
    \sum_{t_1=1}^t (\alpha_{t_1}-\wt\alpha_{t_1}) \diamond \qq^{k,t_1}
    \\
    &\simeq 
    0.
\end{align*}
Moreover, Lemma~\ref{lem:4} (point 4) shows that $\cP^{\perp}_t(\tbA^{(k)})\{\f_t\}\simeq\tbA^{(k)}\{(\f_{t}^{\otimes k-1})_{\perp}\}$ 
has entries which are approximately independent Gaussian with variance 
\[
    \sigma^2_{t,s}\equiv 
    \lt(
    \Gamma^{(k),\nabla}
    \diamond
    (\f_{t}^{\otimes k-1})_{\perp}
    ,\,
    \Gamma^{(k),\nabla}
    \diamond
    (\f_{t}^{\otimes k-1})_{\perp}
    \rt)
\]
on coordinates $i\in\cI_s$, even conditionally on $\cF_t$. 
Therefore
\begin{align}
\label{eq:ReprSE}
  \qq^{k,t+1} &\ed \sum_{t_1=1}^t \wt\alpha_{t_1} \diamond\qq^{k,t_1} +\sigma_t \diamond \bg + \bferr^{k,t+1}\, ,
\end{align}
where $\|\bferr\|_N\simeq 0$ and $\bg\sim\normal(\bfzero,\id_N)$ is independent of everything else.
It now remains to verify that this agrees with the desired covariance.
As proved in the previous point, for any $t_1\le t$,
\begin{align*}
  R\lt(\qq^{k,t+1} ,\qq^{k,t'+1}\rt)_s
  &\simeq 
  \xi^{k,s}\lt(
  \f_t,\f_{t'}
  \rt)
  \\
  &\simeq
  \E\big[ 
  U^{k,t+1}_s U^{k,t'+1}_s
  \big]\, .
\end{align*}
In particular this establishes convergence of the second moment, so in order to  prove \eqref{eq:ConvergenceLAMP} it is sufficient to establish weak convergence. Hence we may assume $\psi:\bbR^{D \times (t+1)}\to\bbR$ is Lipschitz (rather than just pseudo-Lipschitz).

Using the representation \eqref{eq:ReprSE}, and focusing for simplicity on a single $k$, we get
\begin{align*}
  \frac{1}{N_s}
  \sum_{i\in\cI_s}
  \psi\big(\qq_i^{k,\le t}, q_i^{k,t+1}
  ;
  \be^{\leq t}_i
  \big) 
  &\simeq 
  \frac{1}{N_s}
  \sum_{i\in\cI_s}
  \psi\lt(
      \qq_i^{k,\le t},
      \sum_{s=1}^t 
      \wt\alpha_s \qq^{k,s} 
      +
      \sigma_tg_i
      ;
    \be^{\leq t}_i
  \rt)
  \\
  &\simeq 
  \frac{1}{N_s}
  \sum_{i\in\cI_s}
  \bbE^{g\sim\cN(0,1)}
  \psi\lt(
  \qq_i^{k,\le t},
  \sum_{s=1}^t 
  \wt\alpha_s 
  \qq^{k,s} 
  +
  \sigma_t g
  ;
  \be^{\leq t}_i
  \rt)
  \, .
\end{align*}
The second equality above follows by Gaussian concentration since $\psi$ is assumed Lipschitz. Applying the induction hypothesis now implies \eqref{eq:ConvergenceLAMP}, except that $\be^{t+1}$ is not present. 
However since $\be^{t+1}_i$ and $E^{t+1}_s$ have the same law and are both independent of the past, $\bbW_2$ convergence immediately transfers by Proposition~\ref{prop:W2-converge-augment} below. This completes the proof of part $(c)$.

\begin{proposition}
\label{prop:W2-converge-augment}
    Let $\nu_n=\sum_{i=1}^n \delta_{\wh X_i}$ for $n\geq 1$ be a sequence of probability measures on $\bbR^k$ converging to $\nu\in\cP(\bbR^k)$ in $\bbW_2$. 
    Let $\mu\in\cP(\bbR^k)$ be a probability distribution with finite second moment. Let 
    \[
    E_1,\dots,E_N\stackrel{i.i.d.}{\sim}\mu
    \]
    and set
    \[
    \wt\nu_n=\sum_{i=1}^n \delta_{(\wh X_i,E_i)}.
    \]
    Then
    \[
    \plim_{N\to\infty}
    \bbW_2(\wt\nu_N,\nu\otimes\mu)=0.
    \]
\end{proposition}

\begin{proof}
    Using Proposition~\ref{prop:W2-converge-basic} applied to $\nu$, we can find for any $\eps>0$
    a coupling $\Pi=\big((\wh X_i,X_i)\big)_{i\in [N]}$ of $\nu_n$ with i.i.d. samples $\hat\nu_n$ with transport cost at most $\eps$. Generate independent variables $E_1,\dots,E_N\stackrel{i.i.d.}{\sum}\mu$. Then note that 
    \begin{align*}
    \bbW_2\big(\wt\nu_n,\nu\otimes \mu\big)
    &\leq
    \bbW_2\lt(\wt\nu_n,\sum_{i=1}^n \delta_{(X_i,E_i)}\rt)
    +
    \bbW_2\lt(\sum_{i=1}^n \delta_{(X_i,E_i)},\nu\otimes \mu\rt)
    \\
    &\leq
    \eps + o_{\bbP}(1)
    .
    \end{align*}
    Here in the latter step we used the assumption on the coupling $\Pi$ for the first term and Proposition~\ref{prop:W2-converge-basic} applied to $\nu\otimes \mu$ on the second term. This completes the proof.
\end{proof}

\subsection{Asymptotic Equivalence of AMP and Long AMP}

Here we show that AMP and LAMP produce approximately the same iterates. 
\begin{lemma}
\label{lem:ampequalslamp}
Let $\{\bG^{(k)}\}_{k\le D}$ be standard Gaussian tensors, and $\bA^{(k)} = \Gamma^{(k)}\diamond \bG^{(k)}$ for $k\ge 2$. Consider the corresponding AMP
iterates $\ZZ_{t}\equiv (\bz^{k,t_1})_{k\le D,t_1\le t}$ and LAMP iterates $\bQ_{t}\equiv (\qq^{k,t_1})_{k\le D,t_1\le t}$,
from the same initialization $\ZZ_0=\bQ_0$ satisfying the assumptions of Theorem \ref{thm:mixedAMP}
and Theorem \ref{thm:SELAMP}.

Let $\f_t = f_t(\VV_t;\bE_t)$, $t\ge 0$ be the nonlinearities applied to LAMP iterates.
Further assume that there exists a constant $C<\infty$ such that, for all $t\le T$,
\begin{itemize}
\item[$(i)$] The LAMP Gram matrices $\bG_{k,t} = \bG_{k,t}$ are
  well-conditioned as guaranteed by Assumption~\ref{as:well-conditioned}, i.e., 
 \[
    C^{-1}\le \sigma_{\min}(\bG_{k,t})\le \sigma_{\max}(\bG_{k,t})\le C,
    \quad\quad
    \forall k\le D,~t\le T\,.
\]
\item[$(ii)$] Let the linear operator $\cT_{k,t}:\bbR^{N\times t}\to\bbR^{N\times t}$ be defined as per \eqref{eq:Tdef}, with $\bG_{k,t} = \bG_{k,t}(\VV)$,
and $\f_t=f_t(\VV,\bE_t)$, and define $\cL_{k,t} = \bone+\cT_{k,t}$. Then 
\[
    C^{-1}\le \sigma_{\min}(\cL_{k,t})\le \sigma_{\max}(\cL_{k,t})\le C.
\]
\end{itemize}
Then, for any $t\le T$, we have
\begin{align}
\|\ZZ_{t} - \bQ_{t}\|_N\simeq 0 \, .
\end{align}
\end{lemma}
\begin{proof}
Throughout the proof we will suppress $\bE_t$ and simply write $f_t(\bW_t)$ or $f_t(\bV_t)$ to distinguish AMP and LAMP iterates, and analogously for
$\bG_{k,t}(\bW_t)$ or $\bG_{k,t}(\bV_t)$.
The proof is by induction over the iteration number, so we will assume it to hold at iteration $t$, 
and prove it for iteration  $t+1$. We prove the induction step by establishing the following two facts for each $2\leq k\leq D$:
\begin{align}
\big\|\AMP_{t+1}(\ZZ_{t})_k -\AMP_{t+1}(\bQ_{t})_k \big\|_N&\simeq 0 \, ,\label{eq:LAMPapprox1}\\
\big\|\AMP_{t+1}(\bQ_{t})_k -\LAMP_{t+1}(\bQ_{t})_k \big\|_N &\simeq 0\, . \label{eq:LAMPapprox2}
\end{align}

Let us first consider the claim \eqref{eq:LAMPapprox1}, and note that
\begin{align*}
    \AMP_{t+1}(\ZZ_{t})_k -\AMP_{t+1}(\bQ_{t})_k  &= \bA^{(k)}\{f_t(\bW_t)\}-\bA^{(k)}\{f_t(\bV_t)\} 
    \\
    &\quad - 
    \sum_{t_1\leq t} 
    d_{t,t_1,k}
    \diamond
    \big(f_{t_1-1}(\bW_{t_1-1})-f_{t_1-1}(\bV_{t_1-1})\big)\, ,
\end{align*}
where we wrote $d_{t,t_1,k,s}$ for the coefficients of \eqref{eq:AMP-def2}, with AMP iterates replaced by LAMP iterates. 
We then have 
\begin{align*}
    \big\|\AMP_{t+1}(\ZZ_{t})_k -\AMP_{t+1}(\bQ_{t})_k \big\|_N 
    &\le 
    D_{1,t}+D_{2,t}\, ;
    \\
    D_{1,t} 
    & \equiv 
    \big\| \bA^{(k)}\{f_t(\bW_t)\}-\bA^{(k)}\{f_t(\bV_t)\}\big\|_N\, ,
    \\
    D_{2,t} 
    &\equiv 
    \sum_{t_1\leq t,~s\in\sS} 
    |d_{t,t_1,k,s}| 
    \cdot 
    \big\|f_{t_1-1,s}(\bW_{t_1-1})- f_{t_1-1,s}(\bV_{t_1-1})\big\|_N\, .
\end{align*}
Notice that, by the induction assumption (and recalling that each $f_{t,s}$ is Lipschitz continuous and acts component-wise): 
\begin{equation}
\label{eq:InductionF}
    \big\|f_t(\bW_t)-f_t(\bV_t)\big\|_N
    \le C_T
    \sum_{t_1\le t,~k\le D}\|\bw^{k,t_1}-\bv^{k,t_1}\|_N 
    \simeq 
    0
    \, .
\end{equation}
Further, for any tensor $\bT\in(\bbR^{N})^{\otimes k}$, and any vectors $\bv_1,bv_2\in\bbR^N$,
\begin{align}
\|\bT\{\bv_1\}-\bT\{\bv_{2}\}\|_N\le (N^{\frac{k-2}{2}}\|\bT\|_{\op})  (\|\bv_1\|_N+\|\bv_2\|_N)^{k-2} \|\bv_1-\bv_2\|_N
\end{align}
Using Lemma~\ref{lem:4}, this implies that the following bound holds with high probability for a constant $C$:
\begin{align*}
    D_{1,t} 
    &\le C  
    (\|f_t(\bW_t)\|_N+\|f_t(\bV_t)\|_N)^{k-2} \|f_t(\bW_t)-f_t(\bV_t)\|_N\\
    & \le 
    C  (2\|f_t(\bV_t)\|_N+\|f_t(\bW_t) -f_t(\bV_t)\|_N)^{k-2} \|f_t(\bW_t) -f_t(\bV_t)\|_N
    \\
    &\simeq 0.
\end{align*}
The last step follows from \eqref{eq:InductionF} and Theorem~\ref{thm:SELAMP}, which implies (recall each $f_{t,s}$ is Lipschitz) that
$\|f_t(\bV_t)\|_N\le C$ with probability $1-o(1)$. Notice that the same argument implies $\|f_t(\bW_t)\|_N \le C$ with high probability.

Similarly, $D_{2,t}\simeq 0$ follows since $\|f_{t_1-1}(\bW_{t_1-1})- f_{t_1-1}(\bV_{t_1-1})\|_N\simeq 0$ and $|d_{t,t_1,k,s}|\le C_T$ by construction, thus yielding \eqref{eq:LAMPapprox1}.

We now prove \eqref{eq:LAMPapprox2}. Comparing \eqref{eq:AMP-def2} and \eqref{eq:LAMP1}, with
$\cP_t^{\parallel} = \bfone-\cP_t^{\perp}$ we find
\begin{align}
\label{eq:AMP-LAMP}
    \AMP_{t+1}(\bQ_{t})_k 
    -
    \LAMP_{t+1}(\bQ_{t})_k 
    &= 
    \cP_t^{\parallel} (\bA^{(k)})\{f_t(\bV_t)\}
    -
    \ons_{k,t+1}
    -
    \sum_{0\leq t_1\leq t-1}  
    h_{t,t_1,k}\diamond \qq^{k,t_1+1}\, ,
    \\
\nonumber
    \ons_{k,t+1} 
    &= 
    \sum_{t_1\leq t} 
    d_{t,t_1,k}\diamond f_{t_1-1}(\bV_{t_1-1})
\end{align}
Note that $\cP_t^{\parallel} (\bA^{(k)})=\E\lt[\bA^{(k)}|\cF_t\rt]$, where $\cF_t$ here is the analogous $\sigma$-algebra generated by 
$\{\bq^{k,t_1},\be^{t_1}\}_{t_1\le t, k\le D}$. Equivalently, this is the conditional expectation of $\bA^{(k)}$ given the linear constraints
\begin{align}
\A^{(k)}\{f_{t_1}(\bV_{t_1})\}&=\by_{k,t_1+1}\, ,\;\;\;\;\;\mbox{ for } t_1\in \{0,\dots, t-1\}\, ,
\end{align}
Also notice that, by the induction hypothesis, and the definition of $\by_{k,t_1}$, \eqref{eq:Ydef1}, we have for all $t_1\le t$, 
\begin{equation}
\label{eq:Y-ons}
    \by_{k,t_1} \simeq  \qq^{k,t_1}+\ons_{k,t_1}\, .
\end{equation}
Lemma~\ref{lem:symregression} implies that $\cP_t^{\parallel} (\bA^{(k)})$ takes the form of \eqref{eq:SymmRegression} for a suitable matrix
$\hZZ_{k,t}\in\bbR^{N\times t}$. 
The key claim is that
\begin{equation}
\label{eq:ZequalsQ}
    \hZZ_{k,t} \simeq \bQ_t\, .
\end{equation}
In order to establish this claim, we show that,  under the inductive hypothesis,
\begin{equation}
\label{eq:1+TQ=Y}
    (\bfone+\cT_{k,t})\bQ_t\simeq \bY_{k,t}.
\end{equation}
Since $\cL_{k,t}=\bfone +\cT_{k,t}$ is well-conditioned by assumption, the combination of \eqref{eq:ZZ-eq} and \eqref{eq:1+TQ=Y} implies $\hZZ_{k,t}\simeq \bQ_t$. By \eqref{eq:Y-ons}, in order to prove \eqref{eq:1+TQ=Y}, it is sufficient to show that 
\begin{equation}
\label{eq:this-claim}
    \cT_t\bQ_t
    \simeq 
    \ONS_{k,t}
    \equiv 
    [\ons_{k,1}|\cdots|\ons_{k,t}].
\end{equation}

In order to prove \eqref{eq:this-claim}, we use Theorem~\ref{thm:SELAMP}. Recall that 
\begin{align*}
    C_{t_1,t_2,s} &= \E\{U^{k,t_1}_s U^{k,t_2}_s\}\,,
    \\
    W^{t_1}_s &= \sum_{2\leq k\leq D} U^{k,t_1}_s\,,
    \\
    \bC_{\le t} &= (C_{t_1,t_2,s})_{t_1,t_2\le t}\,.
\end{align*}
(The value $2\leq k\leq D$ is implicitly fixed in the definition of $\bC_{\leq t}$.) By Theorem~\ref{thm:SELAMP}, 
\[
    C_{t_1+1,t_2+1} \simeq \<\qq^{k,t_1+1},\qq^{k,t_2+1}\> 
    \simeq 
    (\bG_{\xi^{k,s},t}(\bV))_{t_1,t_2}
    ,
    \quad\forall~t_1,t_2\le t.
\]
This implies for any $0 \le t_1\le t-1$ and $s\in\sS$,
\begin{align}
\nonumber
&   \sum_{t_2=0}^{t-1}
    (\bG_{\xi^{k,s},t-1}^{-1})_{t_1,t_2}
    R\lt(\qq^{k,t_2+1},f_{t-1}(\VV_{t-1})\rt)_s
    \\
\nonumber
    &\simeq 
    \sum_{t_2=0}^{t-1} 
    (\bC_{\le t,s}^{-1})_{t_1+1,t_2+1} 
    \E\lt[
    U^{k,t_2+1}_s f_{t-1,s}(W^0_s,\dots,W^{t-1}_s;E^0_s,\dots,E^{t-1}_s)
    \rt] 
    \\
\label{eq:Stein}
    &= 
    \E\lt[
        \frac{\partial f_{t-1,s}}{\partial W^{t_1+1}_s} (W^0_s,\dots,W^{t-1}_s;E^0_s,\dots,E^{t-1}_s)
    \rt] 
    \bfone_{t_1\le t-2}\, .
\end{align}
Indeed, Gaussian integration by parts yields the latter expression (it can be done conditionally on the variables $E$ since they are independent). Combining \eqref{eq:Stein} with the definition \eqref{eq:AMP-def2} will now allow us to conclude $\cT_{k,t}\bQ_t\simeq \ONS_{k,t}$ as desired. Indeed for each $s\in\sS$ we have
\begin{align*}
    \big[\cT_{k,t}\bQ_t\big]_{t,s} 
    &= 
    \sum_{t_1=0}^{t-1} 
    \partial_{s'}\xi^{k,s}(\vR(\f_{t_1},\f_{t-1}))
    \Big(
    \sum_{t_2=0}^{t-1} 
    (\bG_{\xi^{k,s'},t-1}^{-1})_{t_1,t_2} 
    \,
    R_{s'}(\qq^{k,t_2+1}, \f_{t-1})
    \Big)
    \f_{t_1,s}
    \\
    &\simeq 
    \sum_{t_1=0}^{t-2}
    \sum_{s'\in\sS}
    \partial_{s'}\xi^{k,s}
    \big(
    \vR(\f_{t_1}, \f_{t-1})
    \big)
    \cdot 
    \E\lt[
    \frac{\partial f_{t-1,s'}}{\partial W^{t_1+1}_{s'}} 
    (W^0_{s'},\dots,W^{t-1}_{s'})
    \rt]
    \f_{t_1,s}
    \\
    &= 
    \ons_{k,t}.
\end{align*}
Having established \eqref{eq:ZequalsQ}, we now use the formula \eqref{eq:SymmRegression} for $\cP^{\parallel}_t(\bA^{(k)})=\E\big[\bA^{(k)}|\cF_t\big]$. The result is:
\begin{align}
\label{eq:AParallel}
    \cP^{\parallel}_{t}(\bA^{(k)})\{\f_t\}
    &\simeq 
    \sum_{t_1\leq t}
    \big(\alpha_{t_1}
    \diamond
    \qq^{k,t_1}
    +
    \beta_{t_1}
    \diamond
    \f_{t_1}\big)\, ;
    \\
\nonumber
     \alpha_{t_1,s}
     &\equiv
     \sum_{0\leq t_2\le t-1} 
     (\bG_{\xi^{k,s},t-1}^{-1})_{t_1,t_2} 
     \,
     \xi^{k,s}\big(
     \vR(
     f_{t_2}(\VV_{t_2}),f_{t}(\VV_t)
     )
     \big)
     \, ,
     \\
\nonumber
    \beta_{t_1,s} 
    &\equiv
    \sum_{s'\in\sS}
    \partial_{s'}\xi^{k,s}\big(
    \vR(\f_{t_1},\f_{t} )
    \big)
    \lt(
    \sum_{0\le t_2\leq t-1} 
    (\bG_{\xi^{k,s'},t-1}^{-1})_{t_1,t_2} 
    \,
    R_{s'}( \qq^{k,t_2},\f_{t} )
    \rt) 
    .
\end{align}
On the other hand, using again \eqref{eq:Stein} gives
\begin{align*}
    \sum_{t_1\leq t}\beta_{t_1}\diamond \f_{t_1} 
    &\simeq  
    \sum_{t_1 \le t-1}  d_{t,t_1,k}\diamond \f_{t_1-1}  
    =
    \ons_{k,t+1},
\\
    \sum_{t_1\leq t}
    \alpha_{t_1}
    \diamond
    \qq^{k,t_1} 
    &\simeq 
    \sum_{0\leq t_1\leq t-1}  
    h_{t,t_1,k}
    \diamond
    \qq^{k,t_1+1}.
\end{align*}
We conclude from \eqref{eq:AMP-LAMP} that $\|\AMP_{t+1}(\bQ_{t})_k -\LAMP_{t+1}(\bQ_{t})_k  \|_N\simeq 0$. This concludes the proof.
\end{proof}

\end{document}